\documentclass[11pt]{amsart}
\usepackage[utf8]{inputenc}
\usepackage[T1]{fontenc}
\usepackage[english]{babel}
\usepackage{amsbsy,amssymb,url}
\usepackage{enumitem}

\usepackage{xcolor}
\usepackage{amsmath}
\usepackage{mathrsfs}
\usepackage{mathtools}
\usepackage{amsfonts,amssymb,latexsym,MnSymbol,dsfont}
\usepackage{geometry}
\newtheorem{thm}{Theorem}

\newtheorem{lem}[thm]{Lemma}

\theoremstyle{definition}

\theoremstyle{remark}
\newtheorem{rem}[thm]{Remark}

\numberwithin{equation}{section}
\numberwithin{thm}{section}

\newcommand{\R}{\mathbb{R}}

\newcommand{\EE}{\mathcal{E}}
\newcommand{\RR}{\mathcal{R}}
\newcommand{\NN}{\mathcal{N}}
\providecommand{\abs}[1]{\left\lvert#1\right\rvert}                             
\providecommand{\norm}[1]{\left\lVert#1\right\rVert}    
\newcommand{\scal}[2]{\left\langle #1,#2\right\rangle}
\newcommand{\vm}{v_\mathrm{m}}

\newcommand{\myx}{\mathcal{X}}
\newcommand{\lsh}{\ell_\mathrm{sh}}
\newcommand{\RRsh}{\RR_\mathrm{sh}}
\newcommand{\Csh}{C_\mathrm{sh}}
\newcommand{\actcon}{\mathcal{I}}
\newcommand*{\dd}{\mathop{}\!\mathrm{d}}
\DeclareMathOperator{\inter}{int}
\usepackage{xcolor}
\usepackage{pgf,tikz}
\usetikzlibrary{arrows}
\usetikzlibrary{decorations.markings,decorations.pathmorphing,patterns,positioning}
\usetikzlibrary{tikzmark,fit,shapes.geometric}

\begin{document}

\title[Optimality conditions for crawlers]
{On the optimal control of rate-independent soft crawlers}

\author[Giovanni Colombo]{Giovanni Colombo}
\address[Giovanni Colombo]{Universit\`a\ di Padova, Dipartimento di Matematica ``Tullio Levi-Civita'', 
via Trieste 63, 35121 Padova, Italy}
\email{colombo@math.unipd.it}

\author[Paolo Gidoni]{Paolo Gidoni}
\address[Paolo Gidoni]{Czech Academy of Sciences, Institute of Information Theory and Automation (UTIA), 
pod vodárenskou veží 4, 182 08, Prague 8, Czech Republic}
\email{gidoni@utia.cas.cz}

\thanks{ The research was partially carried out while P.G. was a postdoc at the I.N.d.A.M. Unit of the University of Padova, with a fellowship of the Istituto Nazionale di Alta Matematica in the framework of the MathTech project. 
G.C. is partially supported by the Padua University grant SID 2018 ``Controllability, stabilizability and infimun gaps for control systems'', BIRD 187147, and is affiliated to Istituto Nazionale di Alta Matematica (GNAMPA). 
P.G. is partially supported by the GA\v{C}-FWF project  19-29646L}



\date{\today}
\begin{abstract}
	Existence of optimal solutions and necessary optimality conditions for a controlled version of Moreau's 
	sweeping process are derived. 
	The control is a measurable ingredient of the dynamics and the constraint set is a polyhedron. The novelty consists
	in considering time periodic trajectories, adding the requirement that the control have zero average, 
	 and considering an integral functional that lacks weak semicontinuity. A model coming from
	the locomotion of a soft-robotic crawler, that motivated our setting, is analysed in detail. In obtaining necessary conditions, an improvement of
	the method of discrete approximations is used.
\end{abstract}
\maketitle

\section{Introduction}
Moreau's sweeping process comprises a class of evolution inclusions that model the displacement of a point $x(t)$  dragged in a normal direction
by a moving (convex or mildly non-convex) closed set: see, e.g., the survey paper \cite{CT} and references therein. 
If the point is also subject to an independent dynamics, then the evolution can be seen as a constrained motion, in which the reaction of the constraint
is active. More precisely, the problem is stated as
\begin{equation}\label{gendyn}
\dot{x}(t)\in - \NN_{C(t)}(x(t)) + g(t,x(t))\quad \text{a.e. in $[0,T]$},\quad x(0)=x_0\in C(0). 
\end{equation}
Here $\NN_C(x)$ denotes the normal cone (of convex analysis if $C$ is convex) to $C$ at $x\in C$.
The case where $C$ is independent of time is particularly meaningful, because it is well known that the problem is equivalent to the so called
projected differential equation
\begin{equation}\label{projdyn}
\dot{x}(t) = \pi_{T_C(x)}(g(t,x(t))), \quad x(0)=x_0\in C, 
\end{equation}
where $\pi_{T_C(x)}(y)$ denotes the projection into the tangent cone to $C$ at $x$ of the vector $y$ (see \cite[Sec.~10.1]{AU}).
The equivalence of \eqref{gendyn} and \eqref{projdyn}, when $C(t)\equiv C$, both explains the role of the constraint in the dynamics and 
its intrinsic nonsmoothness (even discontinuity). Indeed, only one normal vector can be taken in \eqref{gendyn}, 
that is the smallest that cancels the (external) normal component of $f$, in order to keep the trajectory inside $C$. This latter fact follows
from the emptiness of the normal, or tangent, cone to $C$ at points outside $C$.
Moreover, observe that the normal cone mapping $x\mapsto \NN_C(x)$ is discontinuous -- actually it has only closed graph -- for two reasons: first because 
$C$ may be nonsmooth, and, second, because in the interior of $C$, if any, $\NN_C(x)=\{ 0\}$, while at boundary points $\NN_C(x)$ contains at least
a half line. A similar type of discontinuity appears in the right hand side of \eqref{projdyn}. However, the monotone character of the normal
cone mapping allows to prove forward-in-time existence (and uniqueness if the ODE $\dot{x}=g(t,x)$ allows so) of solutions to the Cauchy
problem \eqref{gendyn} under usual conditions imposed on $g$.

The simplest control problems involving Moreau's sweeping process occur when a control parameter $u(\cdot)$ appears within $g$: the dynamics then
becomes
\begin{equation}\label{contrsw}
\dot{x}(t)\in - \NN_{C(t)}(x(t)) + g(t,x(t),u(t)), \; u(t)\in U \text{ a.e.,}
\end{equation}
$U$ being an assigned compact set. This paper is devoted to deriving necessary optimality conditions for a particular Bolza problem involving
\eqref{contrsw} together with further requirements on both $x$ and $u$. This problem is motivated by maximizing the displacement 
in the locomotion of a bio-mimetic soft robotic
crawler, whose mathematical model is presented in detail in Section \ref{sec:model}. The robot can be described as a chain of  
$N$ links, each formed by a spring coupled in series with an actuator, whose length is controlled. 
The movement is one dimensional and the evolution is supposed
to be \textit{quasi static}, i.e., the mechanical system is modelled by a force balance law and therefore obeys a first order differential inclusion.
 After quite a few transformations, that are essentially known in the theory of 
\textit{rate independent evolutions}, one arrives to the controlled dynamics \eqref{contrsw}, where the space dimension of the problem 
is the number of links. Since one wants to find an optimal \emph{gait}, namely a periodic actuation to be repeated an arbitrarily 
large number of times, the fixed initial condition on the trajectory is substituted by a $T$-periodicity condition, $T$ being fixed \textit{a priori}. 
Moreover, in the final model the controls turn out to be the derivative of periodic Lipschitz functions, so that the zero mean condition
\begin{equation}\label{zeromean}
\int_0^T u(t)\dd t =0
\end{equation}
must be imposed on feasible controls. Finally, $C$ turns out to be a polyhedron.
The functional to be maximized is an integral functional $J$ involving two terms, the reaction of the constraint
and the cost of the control (that of course appears with a minus sign):
\begin{equation*}
J(x,u):= \int_0^T \Big(f_1\big(g(t,x(t),u(t))-\dot{x}(t)\big) - f_2(t,u(t))\Big)\dd t.
\end{equation*}
Here $f_1(\cdot)$ is Lipschitz and positively homogeneous with degree one and $f_2(\cdot,\cdot)$, for simplicity, is $\mathcal{C}^1$ with respect to $u$.
In our application, the first summand  in the integrand of $J$ is a function of the reaction of the constraint, measuring the
the displacement of the barycentre of the system of springs, while $f_2$ represents the cost of actuating the control.
If, on one hand, it is natural to assume the convexity of $f_2$, on the other hand the derivation of the explicit form of $f_1$ for our model, presented in Section \ref{sec:compf1}, gives a functional that is not  concave down with respect to
$\dot{x}$ and $u$. For instance, in the simple example presented in Section \ref{sec:crawler}, the first summand of the integrand is $\frac{1}{2}|\dot{x}-u|$,
so the integral functional is not weakly upper semicontinuous in $W^{1,2}([0,T];\mathbb{R}^n)$, cf.~also Remark \ref{rem:vm}. Therefore, the direct method cannot be used in order
to ensure the existence of optimal state-trajectory pairs.

The first contribution of the present paper is proving an existence result for
the maximization of $J$ along trajectories of a controlled sweeping process of the type \eqref{contrsw} by imposing a uniform bound on the
total variation of admissible controls, giving pointwise convergence of a maximizing sequence of controls. 
This is a strong assumption, which however seems to be justified by the observation
that optimal controls are expected to be bang-bang with finitely many switchings (see Section \ref{sec:crawler}), or anyway with finite total
variation (see Section \ref{rem:ex}). Moreover, this requirement does not completely trivialize the existence argument, because in order to allow
passing to the limit on $J$ along a maximizing sequence $(x_\ell,u_\ell)$ one needs also the strong convergence of the sequence of derivatives $\dot{x}_\ell$
of the state variable. While for general differential inclusions this is not possible, the particular structure of the sweeping process allows to
overcome this difficulty.

Our second contribution consists of necessary optimality conditions. The analysis of necessary conditions 
in this type of setting does not follow from the classical literature on state constrained
optimal control problems (see, e.g., \cite{vinter}), since the right hand side of the dynamics is not Lipschitz (actually, is very far from being so), 
with respect to the state $x$. There are essentially two ways to attack the problem. The first one is based on a regularization of the dynamics and goes
back essentially to \cite{BK}, see also \cite{arc, DPFS}. It provides an adjoint equation in the sense of measures together with
a maximum principle of Pontryagin type, as it may be expected in such problems, but is -- up to now -- limited by requiring the set $C$
to be smooth. The second one, that is due to Mordukhovich and collaborators
(see, e.g., \cite{CHM, gmd} and references therein), is based on discrete approximations. 
This technique fits perfectly with our polyhedral setting, but provides only
a weaker form of the maximum principle. In this paper we adapt to our problem the method of discrete approximations, by considering periodic trajectories
and adding the control constraint \eqref{zeromean}. Moreover, taking inspiration from the fact that the normal vector in \eqref{contrsw} cannot
be chosen independently of the control, we simplify the discretization procedure by avoiding computing the normal cone to the normal cone $\NN_C(x)$.
Furthermore, our approximation technique allows general measurable controls, not being limited to controls with bounded variation as 
in \cite{gmd}. Finally, we deal with the nonconcavity of the integral functional without relying at all on relaxation arguments. Actually,
in this case relaxation results are difficult to obtain, since the integral functional involves also the derivative of the state, not only the control
variable, and furthermore periodic solutions are considered. Nevertheless, the obtained necessary conditions are very similar to those derived in the framework of \cite{gmd}.

The problem and the main results of the paper are stated in Section \ref{sec:statth}. The existence proof is presented in 
Section \ref{sec:proofex}, while the proof of the theorem on necessary conditions appears in Sections \ref{sec:discappr}, 
\ref{sec:discopt}, \ref{sec:discnec}, and \ref{sec:main-proof}. The intermediate Section \ref{sec:model} contains the general 
derivation of the model, while in Section \ref{sec:crawler} we discuss extensively the necessary condition obtained in Theorem 
\ref{nec} in the case of a one-link crawler, and make a few technical remarks and comments.

\section{Statement of the problem and main results}\label{sec:statth}

\subsection{Notation}
Let $A$ and $S$ be sets, with $A\subset S$. We set, for $x\in S$,
\begin{equation*}
\mathds{1}_A(x)=\begin{cases}
1&\text{if } x\in A\\
0&\text{if } x\not\in A.
\end{cases}
\end{equation*}
The Lebesgue measure of $S\subset \mathbb{R}$ is denoted by $|S|$. Given an integrable function $f$ on a set $S$ with 
finite measure, we denote its average as
\begin{equation*}
\strokedint_S f(s)\dd s=\frac{1}{\abs{S}}\int_S f(s)\dd s
\end{equation*}
The closed unit ball of a normed space $X$ is $B_X$  and 
the interior of a set $S\subset X$ is denoted by $\mathrm{int}\, S$. The 
convergence with respect to the Hausdorff distance between closed subsets of $X$ will be considered (see \cite[Section 4.C]{RW}).
We denote with $\mathcal{C}([0,T],X)$ the space of continuous functions from $[0,T]$ to $X$, endowed with the $\norm{\cdot}_\infty$ norm; 
with $\mathcal{C}^\ast([0,T],X)$ its dual space, and with $\mathcal{C}^\ast_+([0,T],X)$ the 
subset of positive measures. 

Classical constructs of nonsmooth analysis will be used. In particular, for a set $S$, the cone of (limiting/Mordukhovich) normal vectors
to $S$ at $x\in S$ is denoted by $\NN_S(x)$ (see \cite[Definition 6.3]{RW}), while for $x\notin S$, we set $\NN_S(x)=\emptyset$. For a Lipschitz
function $f\colon X \to \R$, the (limiting/Mordukhovich) subdifferential of $x$ is denoted by $\partial f(x)$ 
(see \cite[Definition 8.6 (b)]{RW}); we also refer the interested reader to \cite[Chapter 1]{BM}, where the above concepts 
are used also in the context of \textit{coderivatives} of set-valued maps.

By a process, or a state-control pair, for the controlled dynamics \eqref{contrsw} we mean the couple $(x,u)$, where $x$ is a solution of
\eqref{contrsw} corresponding to the (measurable) control $u$. The total variation of a function $u$ of one real variable is denoted by
$TV(u)$.

\subsection{Statement of the problem} \label{sec:statement} Let $C$ be a given polyhedron in a Euclidean space $X=\R^n$, defined as
\begin{equation}\label{eq:polyhedron}
C=\bigcap_{j=1}^\sigma C_j, 
\end{equation}
where, for suitable unit vectors $x_\ast^j\in X$ and real numbers $c_j$,
\begin{equation*}
C_j := \{ x\in X : \langle x_\ast^j,x\rangle \le c_j\}.
\end{equation*}
Given $x\in C$, let us denote with $\actcon (x)$ the set of active constraints in $x$, namely 
$\actcon (x)=\{j=1,\dots \sigma : \langle x_\ast^j,x\rangle = c_j\}$.
We assume that $C$ has non-empty interior; in other words, the Positive Linear Independence Constraint Qualification 
(PLICQ) holds, i.e., if $\sum_{j\in\actcon(x)} \lambda_j x_\ast^j =0$
and $\lambda_j\ge 0$, $j=1,\ldots ,\sigma$, then $\lambda_j=0$ for all $j$. In this case, the normal cone to $C$ at $x\in C$
is
\[
\NN_C(x) = \left\{ v: v=\sum_{j\in \actcon (x)} \lambda_j x_\ast^j,\,\lambda_j\ge 0\right\}.  
\]
 The following assumptions will be considered.
\begin{itemize}
\item[($H_U$)] The control set $U\subset \R^d$ is a compact and convex set, and that $d\le \mathrm{dim}\, X$. 
Moreover, since we will require a zero-average condition on $u(t)$, we assume $0\in \inter U$.
\end{itemize}
We remark that in order to guarantee the existence of controls $u(t)$ with zero-average it is sufficient to assume $0\in U$, but if the zero lies on the boundary then all the zero-average functions $u(t)$ have values in a lower dimensional convex set $\tilde U$ with  $0\in \inter \tilde U$.

We consider the maps  $g\colon[0,T]\times X\times U\to X$, $f_1\colon X\to \R$ and $f_2\colon [0,T]\times U\rightarrow \R$ with the following properties.
\begin{itemize}
\item[($H_g$)] the map $t\mapsto g(t,x,u)$ is measurable for all $x\in X, u\in U$  and there exists
$L\ge 0$ such that $|g(t,x,u)|\le L$ for a.e.~$t\in [0,T]$ and all $(x,u)\in X\times U$, $(x,u)\to g(t,x,u)$ is smooth
 and there exists $L'\ge 0$ such that $|D_x g(t,x,u)|\le L'$ for a.e.~$t\in [0,T]$ and all $(x,u)\in X\times U$; 
\item[($H_{f_1}$)] the map $x\mapsto f_1(x)$ is Lipschitz continuous;
\item[($H_{f_2}$)] the map $t\mapsto f_2 (t,u)$ is continuous for all $u\in U$ and the map $u\mapsto f_2(t,u)$ is continuously differentiable
 for a.e.~$t\in [0,T]$ and all $u\in U$;

\end{itemize}

\medskip
\noindent\textbf{Problem \textup{(P)}} \emph{We set  $T>0$ and consider the problem
\begin{equation}\label{traj}
\begin{cases}
\dot{x}\in -\NN_{C}(x) + g(t,x,u)\;\text{ a.e.~on~}\; [0,T],\\
u(t)\in U\; \text{a.e. and }\int_0^T u(t) \dd t =0,\\
x(0)=x(T).
\end{cases} 
\end{equation}
We wish to maximize the integral functional
\begin{equation}\label{eq:cost}
J(x,u):=\int_0^T \Big(f_1\left(g(t,x(t),u(t))-\dot{x}(t)\big)- f_2(t,u(t))\right) \dd t 
\end{equation}
among all processes $(x,u)$ of \eqref{traj}.}
\medbreak

\subsection{Statement of the main results}\label{sec:mainres}
The setting of the existence theorem is slightly more general that in the previous section. Therefore we list the assumptions
directly in the statement of the result. We will make reference to problem \eqref{traj}, but the statement can be easily reformulated
for a Cauchy problem, with or without the constraint on the mean of the control.
\begin{thm}\label{th:ex}
Let $C\subset X$ be compact and convex, let $U\subset \mathbb{R}^d$ be compact, and let $K>0$. Let $g\colon [0,T]\times X\times U\to X$
be measurable with respect to $t$, continuous with respect to $(x,u)$ and uniformly bounded. Let $f\colon [0,T]\times X^2\times U\to X$
be measurable with respect to $t$ and upper semicontinuous with respect to $(x,\dot{x},u)$. Set 
\[
\mathscr{U}_K:=\{ u\in L^1(0,T;\mathbb{R}^d): u(t)\in U\text{ a.e. and } TV(u)\le K\}                                                                                             
\]
and assume that the problem \eqref{traj} admits solutions with $u\in\mathscr{U}_K$.
Then the integral functional
\[
\mathscr{I} (x,u):= \int_0^T f(t,x(t),\dot{x}(t),u(t))\dd t
\]
admits a maximizer among all processes $(x,u)$ of problem \eqref{traj} such that $u\in\mathscr{U}_K$.
\end{thm}

Our necessary optimality conditions are applicable to local $W^{1,2}$-optimal processes for problem (P). 
We say that $(\bar{x},\bar{u})$ is a local $W^{1,2}$-optimal process for (P) provided there exists 
$\bar{\varepsilon}>0$ such that for all processes $(x,u)$ of \eqref{traj} with $\| x-\bar{x}\|_{W^{1,2}([0,T];X)}+\|u-\bar{u}\|_{L^2([0,T];X)}
<\bar{\varepsilon}$ one has $J(x,u)\le J(\bar{x},\bar{u})$.

\medbreak
The result on necessary optimality conditions is the following
\begin{thm}\label{nec}
	Let the assumptions ($H_U$), ($H_g$), ($H_{f_1}$), ($H_{f_2}$) hold, and let
	$(\bar{x},\bar{u})$ be a local $W^{1,2}$-optimal process for the problem \textup{(P)}. Then there exist
	\begin{itemize}
		\item a number $\lambda \ge 0$,
		\item a function of bounded variation $p\colon[0,T]\to X$,
		\item positive and finite Radon measures $\dd\xi^j$ on $[0,T]$, $j=1,\ldots ,\sigma$,
		\item a function $\psi \in L^1 (0,T; X)$,
		\item a vector $\omega \in B_X$
	\end{itemize}
	that satisfy the following properties:
	\begin{itemize}
		\item (adjoint equation)\quad $\dd p = -D_x g(t,\bar{x}(t),\bar{u}(t))^\ast \dd t+
		\sum_{j=1}^\sigma \dd \xi^j x_\ast^j \quad\text{(in $\mathcal{C}^\ast ([0,T];X)$)}$,
		\item (transversality)\quad $p(T)=p(0)$,
		\item (weak maximality condition)\quad $\psi(t) =
		-D_w g(t,\bar{x}(t),\bar{u}(t))^\ast  p(t) -\omega -\lambda  D_w f_2(t,\bar{u}(t))\in \NN_U (\bar{u}(t))$ a.e.~on $[0,T]$,
		\item (support condition)\quad for all $j=1,\ldots ,\sigma$, 
		$\,\mathrm{supp}\, (\dd \xi^j) \subset \{ t\in [0,T]: j\in\actcon(\bar{x}(t))\}$,
		\item (nontriviality condition)\quad $\lambda + \| p \|_\infty =1$.
	\end{itemize}
\end{thm}
The proof of Theorem \ref{nec} will be carried out in Sections \ref{sec:discappr} -- \ref{sec:main-proof}.

\begin{rem}
	One can consider a dynamics more general than \eqref{traj}, namely
	\begin{equation*}
	\begin{cases}
	\dot{x}(t)\in -\NN_{C}(x(t)) + g(t,x(t),u(t))\;\text{ a.e.~on~}\; [0,T],\\
	\dot{y}(t) = f_1(g(t,x(t),u(t))-\dot{x}(t)) - f_2(t,y(t),u(t)),\\
	u(t)\in U\; \text{ and }\;\int_0^T u(t) \dd t =0,\\
	x(0)=x(T), \quad y(0)=0
	\end{cases} 
	\end{equation*}
 The object to be maximized in this case is 
	\begin{equation*}
	\varphi(y(T)). 
	\end{equation*}
	for a suitable (e.g., u.s.c.) function $\varphi$.
	This amounts to adding only technical difficulties, that we wish to avoid here.
\end{rem}

\section{A motivating locomotor model}\label{sec:model}
\subsection{Introduction}
In the last years, an increasing attention has been directed to the analysis, control and optimization of the locomotion 
of simple devices, such as a chains of linked segments or blocks. The same trend can be observed both in swimming 
\cite{BMZ,MdS,ZC} or in locomotion on a solid surface, such as inching and crawling \cite{Ago,BPZZ,Gid18,WL}.
The employment of very simple mechanisms has two main motivations. The first one is that 
a simple mechanism allows an easier miniaturization of the device. 
The second advantage comes from the paradigm of \emph{simplexity} in soft robotics, based on the idea that a 
simple mechanism with a low number of control parameters may still achieve a complex behaviour and adaptability 
to an unknown environment by exploiting the large deformation of a soft, elastic body \cite{LM}. 
This also motivates the strong role played by elasticity in our model, despite introducing several additional mathematical challenges.

In the specific case of crawling locomotion, several approaches have been applied to the search of optimal gaits. 
One strategy is to consider suitable approximations in the model, for instance neglecting elasticity or working on a 
small deformation regime, so that, with a certain degree of approximation, it is possible to have an explicit description 
of the dynamics in terms of the control function \cite{Ago,DesTat}.
Another approach, introducing a feedback mechanism in order to apply adaptive control, is presented in \cite{Beh1,Beh2}. 
A model-free control framework, based on the decomposition of possible gaits as paths between a finite number of basic states, 
has been proposed in \cite{vikas}.

In this paper we present a more mathematical approach, based on a maximum principle of Pontryagin type. 
On one hand this, compared to the more pragmatical approaches mentioned above, makes more difficult to obtain 
an explicit characterization of optimal gaits. On the other hand, we believe that the development of a more 
theoretical approach, in parallel to engineering studies appearing in literature, may contribute to a better 
understanding of the challenging issues raised by the  optimal control of a soft bodied locomotor.  
In our opinion crawling locomotion is  not only an interesting problem \emph{per se}, but represent a less hostile 
framework in which we can learn to unravel difficult phenomena that appear in a more general setting.

Concerning the specific model considered in our paper, we will follow the approach developed in \cite{Gid18,GdS1}. 
Our choice is motivated by the fact that such class of model includes the two main features observed in crawlers 
(a stick-slip interaction with the environment and an elastic body) without adding unnecessary elements. 
Moreover, even if here we consider only a smaller family of cases, the same formalism of sweeping process applies 
to a large class of behaviour, including continuous bodies and time-dependent friction \cite{Gid18}, 
opening the way for future developments of our results.

\begin{figure}
	\begin{center}
		\begin{tikzpicture}[line cap=round,line join=round,>=triangle 45,x=1.0cm,y=1.0cm, line width=1.1pt,scale=1.2]
	\clip(0.,-0.5) rectangle (11.,2.);
	\draw[line width=0.8pt] (1.,1.)-- (1.6,1.);
	\draw[line width=0.8pt] (1.9,1.)-- (2.5,1.);
	\draw[line width=0.8pt] (1.6,0.9)-- (1.6,1.1);
	\draw[line width=0.8pt] (1.9,0.95)-- (1.9,1.05);
	\draw[line width=0.8pt] (1.65,0.95)-- (1.9,0.95);
	\draw[line width=0.8pt] (1.65,1.05)-- (1.9,1.05);
	\draw[line width=0.8pt] (1.6,1)-- (1.85,1);
	\draw[line width=0.8pt] (1.6,1.1)-- (1.85,1.1);
	\draw[line width=0.8pt] (1.6,0.9)-- (1.85,0.9);
	\draw[line width=0.8pt] (4.,1.)-- (4.6,1.);
	\draw[line width=0.8pt] (4.9,1.)-- (5.5,1.);
	\draw[line width=0.8pt] (4.6,0.9)-- (4.6,1.1);
	\draw[line width=0.8pt] (4.9,0.95)-- (4.9,1.05);
	\draw[line width=0.8pt] (4.65,0.95)-- (4.9,0.95);
	\draw[line width=0.8pt] (4.65,1.05)-- (4.9,1.05);
	\draw[line width=0.8pt] (4.6,1)-- (4.85,1);
	\draw[line width=0.8pt] (4.6,1.1)-- (4.85,1.1);
	\draw[line width=0.8pt] (4.6,0.9)-- (4.85,0.9);
		\draw[line width=0.8pt] (7.,1.)-- (7.6,1.);
	\draw[line width=0.8pt] (7.9,1.)-- (8.5,1.);
	\draw[line width=0.8pt] (7.6,0.9)-- (7.6,1.1);
	\draw[line width=0.8pt] (7.9,0.95)-- (7.9,1.05);
	\draw[line width=0.8pt] (7.65,0.95)-- (7.9,0.95);
	\draw[line width=0.8pt] (7.65,1.05)-- (7.9,1.05);
	\draw[line width=0.8pt] (7.6,1)-- (7.85,1);
	\draw[line width=0.8pt] (7.6,1.1)-- (7.85,1.1);
	\draw[line width=0.8pt] (7.6,0.9)-- (7.85,0.9);
	
	\draw[line width=0.8pt] (1,1.45)-- (2.5,1.45);
	\draw[line width=0.8pt] (1,1.4)-- (1,1.5);
	\draw[line width=0.8pt] (2.5,1.4)-- (2.5,1.5);
	\draw[line width=0.8pt] (4,1.45)-- (5.5,1.45);
	\draw[line width=0.8pt] (4,1.4)-- (4,1.5);
	\draw[line width=0.8pt] (5.5,1.4)-- (5.5,1.5);
		\draw[line width=0.8pt] (7,1.45)-- (8.5,1.45);
	\draw[line width=0.8pt] (7,1.4)-- (7,1.5);
	\draw[line width=0.8pt] (8.5,1.4)-- (8.5,1.5);
	\draw (1.75,1.45) node[anchor=south] {$L_1 (t)$};
	\draw (4.75,1.45) node[anchor=south] {$L_2 (t)$};
	\draw (7.75,1.45) node[anchor=south] {$L_3 (t)$};
	\draw (3.25,1.1) node[anchor=south] {$k$};
	\draw (6.25,1.1) node[anchor=south] {$k$};
	\draw (9.25,1.1) node[anchor=south] {$k$};
	\draw (1.,1.)-- (1.,0.) node[anchor=north]{$\myx_1(t)$};
	\draw[decoration={aspect=0.5, segment length=2mm, amplitude=2mm,coil},decorate] (2.5,1.)-- (4.,1.);
	\draw (4.,1.)-- (4.,0.)node[anchor=north]{$\myx_2(t)$};
	\draw [decoration={aspect=0.5, segment length=2mm, amplitude=2mm,coil},decorate](5.5,1.)-- (7.,1.);
	\draw (7.,1.)-- (7.,0.)node[anchor=north]{$\myx_3(t)$};
		\draw [decoration={aspect=0.5, segment length=2mm, amplitude=2mm,coil},decorate](8.5,1.)-- (10.,1.);
	\draw (10.,1.)-- (10.,0.)node[anchor=north]{$\myx_4(t)$};
	\draw (0.5,0.)-- (10.5,0.);
	\draw [fill=black] (1.,1.) circle (1pt);
	\draw [fill=black,line width=0.6pt] (1.,0.) --  (1,0.15) -- (1.1,0.15) --(1,0);
	\draw [fill=black,line width=0.6pt] (4.,0.) --  (4,0.15) -- (4.1,0.15) --(4,0);
	\draw [fill=black,line width=0.6pt] (7.,0.) --  (7,0.15) -- (7.1,0.15) --(7,0);
	\draw [fill=black,line width=0.6pt] (10.,0.) --  (10,0.15) -- (10.1,0.15) --(10,0);
	\draw [fill=black] (10.,1.) circle (1pt);
	\draw [fill=black] (7.,1.) circle (1pt);
	\draw [fill=black] (4.,1.) circle (1pt);
	\fill [pattern = north east lines] (0.5,0) rectangle (10.5,-0.1);
	\end{tikzpicture}
	\end{center}\caption{A model of soft crawler.}
\label{fig:crawler}
\end{figure}
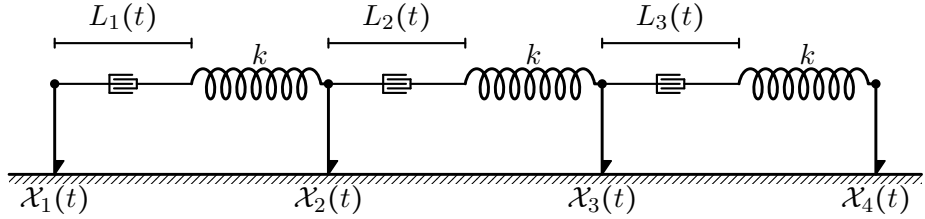

\subsection{A rate-independent model of soft crawler}\label{sec:rateind}
Let us consider the mechanical system illustrated in Figure \ref{fig:crawler}, consisting of a chain of $N$ blocks. 
Each couple of adjacent blocks is joined by a link composed by a spring in series with an actuator, namely an 
element of prescribed length $L_i(t)$, which is our control on the system. The body of the crawler can be therefore 
identified in the reference configuration by a set of $N$ points $\{\xi_1,\dots,\xi_N\}$.
We represent the state of the crawler in the deformed configuration with a vector $\myx=(\myx_1,\myx_2,\dots,\myx_N)$ in $\R^{N}$, 
where $\myx_i$ stands for the displacement of the point $\xi_i$.

We consider the locomotion of our model in the regime of very slow (quasi-static) actuation, so that inertial forces can be
neglected. Hence, the dynamics is described by a force balance between the friction forces acting on the body of the crawler and 
internal elastic forces associated to the deformations of the springs in the links, which can be written as
\begin{equation} \label{eq:forcebalance}
-D_\myx\EE(t,\myx)\in \partial_{\dot \myx} \RR(\dot \myx)
\end{equation}
Here $\EE(t,\myx)$ is the elastic energy of the crawler, and therefore can be expressed as the sum of the elastic energies 
$\EE_i(t,\myx)$ of each link, namely
\begin{equation}\label{eq:linkenergy}
\EE(t,\myx)=\sum_{i=1}^{N-1}\EE_i(t,\myx)=\sum_{i=1}^{N-1}\frac{k}{2}(\myx_{i+1}-\myx_i-L_i(t))^2
\end{equation}
We assume that the actuation functions $L_i\colon[0,T]\to \R$ are Lipschitz continuous. The constant $k>0$ is the elastic 
constant of the springs. Note that the same mathematical structure holds if we replace each actuator with an active control 
on the rest length of the corresponding spring, which is the case of robotic crawlers actuated e.g. by nematic elastomers \cite{dSGN,Jung}.

Each of the points $\xi_i$ is subject to an anisotropic dry friction, so that we can write the friction force $F_i$ acting on $\xi_i$ as
\begin{equation*}
	F_i=F_i(\dot \myx_i)\in\begin{cases}
	\{\mu_i^-\}  &  \text{if $\dot \myx_i< 0$} \\
	[-\mu_i^+,\mu_i^-]  &  \text{if $\dot \myx_i= 0$} \\
	\{-\mu_i^+\} &  \text{if $\dot \myx_i> 0$}
	\end{cases}
\end{equation*}
for some positive coefficients $\mu_i^\pm$. Hence, friction forces can be expressed variationally in \eqref{eq:forcebalance} 
as the subdifferential of the dissipation potential
\begin{equation} \label{xeq:dissipoint}
\RR(\dot \myx)=\sum_{i=1}^{N} \RR_i(\dot \myx_i) \qquad \text{with}\quad  \RR_i(\dot \myx_i)=\begin{cases}
-\mu_i^-\dot \myx_i  &  \text{if $\dot \myx_i\leq 0$} \\
\mu_i^+\dot \myx_i &  \text{if $\dot \myx_i\geq 0$}
\end{cases}
\end{equation}
We recall that $\partial_{\dot \myx}\RR(\myx)\subseteq \partial_{\dot \myx}\RR(0)$
since the function $\RR$ is positively homogeneous of degree one, and set
\begin{equation}\label{eq:defC0}
C_0:=\{\myx\in X : -\mu_i^-\leq\scal{e_i}{\myx}\leq \mu_i^+ \quad \text{for $i=1,\dots,N$}\}=\partial_{\dot \myx}\RR(0)
\end{equation}
where $e_1,\dots,e_n$ denotes the canonical base of $X$.
Since the friction forces in \eqref{eq:forcebalance} are bounded, we cannot allow too large initial elastic forces, 
hence we introduce the \emph{admissibility condition} for the initial state:
\begin{equation}
-D_\myx\EE(t,\myx_0)\in C_0
\end{equation}
In order to guarantee existence and uniqueness of solution for the Cauchy problem with an admissible initial state, 
we make the following assumption: for every subset of indices $J\subseteq\{1,\dots,N\}$ we have
\begin{equation}\label{eq:uniqueness}
\sum_{i\in J} \mu^+_i-\sum_{i\in J^c} \mu^-_i\neq 0 
\end{equation}
where $J^c$ denotes the complement of $J$. We refer to \cite[Section 2]{Gid18} for a complete proof and discussion. 
Since we will refer also later on to results obtained there, we observe for the reader's convenience that the coordinate 
$x$ and the sets $C,\Csh$ in \cite{Gid18} correspond respectively to $\myx,C_0,C$ in this paper.

In order to study the locomotion of our model, it is useful to introduce the projections:
\begin{equation}
\begin{array}{l}
y=\pi_Y(\myx):=\frac{1}{N}\sum_{i=1}^N \myx_i\in\R\\[2mm]
z=\pi_Z(\myx):=(\myx_2-\myx_1,\dots,\myx_N-\myx_{N-1})=:(z_1,\dots, z_{N-1})\in Z\cong \R^{N-1}
\end{array}
\end{equation}
In this way we can split the state of the crawler into two components: the term $y$ describes the position of the crawler, 
whereas $z$ describes its shape, namely the lengths of the $N-1$ links in the deformed configuration.

Setting, without loss of generality $0=y(0)=\pi_Y(\myx(0))$, our problem consists of finding suitable choices of the 
actuations $L_i$ that maximize $y(T)=\pi_Y(\myx(T))$.

\subsection{Formulation as a sweeping process}\label{sec:formsw}

We now show how we can pass from the dynamics \eqref{eq:forcebalance} for the model presented above to a sweeping 
process of the form \eqref{traj} and discuss the other elements of problem (P).

We observe that, since the elastic energy $\EE$ is invariant for rigid translations, it depends only on the shape $z$, namely
\begin{equation}\label{eq:linkenergy2}
\EE(t,\myx)=\scal{k z-\lsh(t)}{z}+\text{time-dependent term}\qquad
\end{equation}
where we define $\lsh(t)=\bigl(kL_1(t),\dots,kL_{N-1}(t)\bigr)$. 
The last term disappears in the dynamics \eqref{eq:forcebalance}, 
so can be neglected for our purposes.

We can reformulate the dynamics \eqref{eq:forcebalance} as the variational inequality
\begin{equation}
\scal{k \pi_Z(\myx(t))-\lsh (t)}{\pi_Z(u-\dot \myx(t))}+\RR(u)-\RR(\dot \myx(t))\geq 0 \quad \text{for every $u\in X$} 
\label{eq:VI}
\end{equation}
cf.~\cite{MieThe,MieRou}. It is easily verified that a function $\myx(t)$ satisfies \eqref{eq:VI} only if its projection $z(t)=\pi_Z(\myx(t))$ satisfies:
\begin{equation}
\scal{k z(t)-\lsh (t)}{w-\dot z(t)}+\RRsh (w)-\RRsh (\dot z(t))\geq 0 \quad \text{for every $w\in Z$} 
\label{eq:RVI}
\end{equation}
where the dissipation potential $\RRsh\colon Z\to \R$ is defined as
\begin{equation} \label{eq:RRsh}
\RRsh (z)=\min \bigl\{\RR(\myx):\myx\in X, \pi_Z(\myx)=z\bigr\}
\end{equation}
The potential $\RRsh$ is convex and positively homogeneous of degree one \cite[Lemma 2.1]{Gid18}.

We notice that, once \eqref{eq:RVI} is solved, the solution to \eqref{eq:VI} can be recovered straightforwardly. 
Indeed, we observe that \eqref{eq:uniqueness}
allows to define a function $\vm\colon Z\to \R$ as the unique satisfying 
\begin{equation}
\RRsh(\pi_Z(\myx))=\RR(\myx) \qquad \text{if and only if}\qquad  \pi_Y(\myx)=\vm(\pi_Z(\myx))  
\end{equation}
cf.~\cite[Lemma 3.2]{Gid18}. 
This property allows us to recover the evolution of $y(t)$ from that of $z(t)$, as 
\begin{equation} \label{eq:Yrecover}
\dot y(t)=\pi_Y(\dot \myx(t))=\vm(\pi_Z(\dot \myx(t)))=\vm(\dot z(t))
\end{equation}

We can now reformulate the problem for the shape coordinates  \eqref{eq:RVI} in the differential inclusion formulation analogue 
to \eqref{eq:forcebalance}, namely
\begin{equation} \label{eq:Rfb}
-k z +\lsh(t)\in \partial_{\dot z}\RRsh(\dot z)
\end{equation}
Let us denote by  $\RRsh^*$ the Legendre transform of $\RRsh$. Setting $C:=\partial_{\dot z}\RRsh(0)$, by the Legendre-Fenchel equivalence we obtain 
\begin{align} \label{eq:Rincl}
\dot z\in \partial_{\zeta}\RRsh^*(-k z +\lsh(t))=\NN_{C}(-k z +\lsh(t))
\end{align}
We observe that $C$ is a polyhedron in $Z$ of the form \eqref{eq:polyhedron}; indeed, by \cite[Lemma 2.2]{Gid18} we deduce  that
\begin{equation}
C=\{z\in Z : -\mu_i^-\leq\scal{\pi_Z(e_i)}{z}\leq \mu_i^+ \quad \text{for $i=1,\dots,N$}\}
\end{equation}
where $e_1,\dots,e_N$ denotes the canonical base of $\R^N$.

Let us now consider the change of variables $x(t)=-kz(t)+\lsh(t)$ and set $u(t):=D_t\lsh (t)$. 
The locomotion of our system, by \eqref{eq:Yrecover} and \eqref{eq:Rincl}, is described by
\begin{equation}\label{eq:cr_dyn}
		\begin{cases}
	\dot{x}(t)\in -\NN_{C}(x(t)) + u(t)\\
	\dot{y}(t) = \vm(u(t)-\dot x(t))
	\end{cases} 
\end{equation}

\subsection{Formulation of the control problem (P)}\label{sec:formcontr}
Now that we have shown how the locomotion of our system can be described by the dynamics \eqref{eq:cr_dyn}, we  
discuss the cost functional and the constraints required in our control problem (P).

A locomotion strategy, be it for crawling, swimming, walking or running, usually can be identified with a gait, namely a periodic pattern  
that is repeated a large number of times. Let us denote with $T$ the period of a gait. 
In our model, this correspond to assume that the function $L_i$ are $T$-periodic, which in terms of the control $u$ in \eqref{eq:cr_dyn} reads
\begin{equation*}
\int_0^T u(t) \dd t=0
\end{equation*}
It is also reasonable to assume that there are some constraints on the speed at which the shape change occurs, corresponding to 
a uniform Lipschitz constant for all the admissible actuations. This, for the dynamics \eqref{eq:cr_dyn}, 
is exactly the constraint $u(t)\in U$, where the set $U$ is of the form 
\begin{equation} \label{eq:Uform}
U=\Pi_{\ell =1}^d [-a_\ell, a_\ell]\subset\mathbb{R}^d ,
\end{equation}
where  $a_\ell >0$ for all $\ell=1,\ldots,d$. 

Since we are considering locomotion model in one-dimension, we want to maximize the advancement of the crawler 
produced by the chosen gait, plus possibly subtracting a cost for the actuation.

Due to the hysteretic behaviour of the sweeping process, a periodic input (in the sense above) does not necessarily produce a 
periodic change in the shape coordinates $w$, and the produced displacement $y(T)-y(0)$ of the crawler in a period depends on its initial shape $z(0)$.
Hence the optimal gait may depend on the ability to exploit a specific initial state, and on the exact number of periods we are considering.  
This does not suit our purposes, since we are interested in an arbitrarily long time behaviour.
However it is know that sweeping processes with a periodic input converge asymptotically to a periodic output
\cite{Kre,Mak,Mak2,CGV}.  
This has already been observed in the models with one link \cite{GdS1} and two links \cite{GdS2}, also noticing that in 
the specific case of some \lq\lq common sense" gaits the convergence to the asymptotic periodic orbit occurs within the first period.

Since we are interested in the long time behaviour, we can therefore optimize on the possible limit cycles 
(for the shape $w$) associated to a given gait, and evaluate the cost functional on a single period. 
This corresponds to optimize on the trajectories that satisfy the periodicity condition
\begin{equation*}
	x(0)=x(T)
\end{equation*}
Regarding the cost functional, denoting with $f_2(t,u)$ a possible cost of the actuation, we have
\begin{align*}
	J(x,u)&=
y(T)-y(0)-\int_0^Tf_2(t,u(t))\big)\dd t \\	&=\int_0^T \big(\vm(u(t)-\dot{x}(t))- f_2(t,u(t))\big) \dd t 
\end{align*}
that corresponds to \eqref{eq:cost} for $g(t,x,u)=u$ and $f_1=\vm$.

The reader may be wondering why we are considering a periodic fixed time problem, instead of a free time or minimum time problem. 
The main motivation is that a gait works as an universal strategy, but might be slightly suboptimal for a specific prescribed problem.
For instance, we expect that some case-by-case tuning on the first and last iterations of the gait, breaking periodicity,  
may provide a minor improvement to the solution. The natural way to avoid these complications is to optimize on the 
limit cycles of the system, as argued above.
The reader may then wonder why to consider a fixed period. From the examples discussed in \cite{dSGN} and \cite{GdS1}, 
one may observe that the action of the actuation of the crawler can be divided into two kinds of effects: 
a change in the tensions of the links ($\dot x\neq 0$) and movement of the contact points ($u-\dot x\neq 0$). 
A sufficiently complex change in the tension is necessary to reach suitable configurations that allow to move each contact point, 
and represents a sort of \lq\lq fixed cost\rq\rq\ necessary for locomotion. Every part of the period that is not used for 
the necessary tension change is best spent in pure locomotion $(u\in \NN_{C}(w))$. A larger period hence increases the 
ratio of the time that is used for locomotion, leading to a better strategy. Hence, we do not expect a free period to 
affect in a significant way the qualitative structure of optimal solution.
We also notice that a fixed period, combined with the bounds \eqref{eq:Uform} on the rate of shape change, implicitly provide 
a bound on the maximum contraction and elongation of each link. This represents an additional physical constraint that should 
be incorporated in our model if such an assumption is removed.
Since, as we will see, our problem is already quite complex as it is, we think that our approach is a preferable 
first step in the study of optimality for crawling locomotion. 

\subsection{Computation of the function $\vm$}\label{sec:compf1}
We now compute explicitly the function $\vm$.
Let us consider $\myx\in C_0$ and denote with $J(\myx)$ the set of active constraints in $\myx$, namely
\begin{equation*}
J(\myx)=\bigl\{j\in\{1,\dots,N\}: \myx_j= -\mu_j^- \bigr\}\cup \bigl\{j\in\{N+1,\dots,2N\}: \myx_{j-N}= \mu_{j-N}^+ \bigr\}
\end{equation*}
(here the number of constraints is $\sigma=2N$).

Let us consider a vector $v\in \R^N$. By the convexity of $C_0$, there exists  $\myx\in C_0$ such that $v\in\NN_{C_0}(\myx)$. 
Moreover we can write, for some non-negative $\lambda_i\geq 0$, $i=1\dots, 2N$,
\begin{equation}\label{eq:decomp}
v=\sum_{i=1}^N \lambda_i e_i+\sum_{i=N+1}^{2N} -\lambda_i e_{i-N} 
\end{equation}
We notice that, since the vector $e_i$ are linearly independent, the coefficients $\lambda_i$ are uniquely determined. 
Moreover, they satisfy the active constraint condition
\begin{equation}\label{eq:active_decomp}
\lambda_i>0 \Rightarrow i\in \actcon(\myx) 
\end{equation}
Let us set $\nu_i=\pi_Z(e_i)$ for $i=1\dots,N$ and $\nu_i=\pi_Z(-e_{i-N})$ for $i=N+1\dots,2N$.
We have
\begin{equation}\label{eq:z_decomp}
\pi_Z(v)=\sum_{i=1}^{2N} \lambda_i \nu_i
\end{equation}
Let us recall that $\pi_Z$ is a linear diffeomorphism between $Z$ and $\pi_Y^{-1}(0)$, and that  $C=\pi_Z(C_0\cap \pi_Y^{-1}(0))$, 
see \cite[Lemma 2.2]{Gid18}. 
We notice that, since $C_0\cap \pi^{-1}_Y(0)$ is a section of a convex set, every vector $z\in Z$ can be written as $z=\pi_Z(v)$, 
with $v\in\NN_{C_0}(\zeta)$ for some $\zeta\in C_0\cap \pi_Y^{-1}(0)$. Moreover, condition \eqref{eq:uniqueness} implies that 
such vector $v$ is unique, see \cite[Lemma 2.3]{Gid18}.
These facts imply that every vector $z\in Z$ admits a decomposition \eqref{eq:z_decomp}, where the coefficients $\lambda_i$ are uniquely determined.

We can finally use such decomposition of $\dot z$ to give an explicit expression for $\vm$:
\begin{equation}\label{eq:vmdec}
\vm\left(\sum_{i=1}^{2N} \lambda_i\nu_i\right)=\pi_Y\left(\sum_{i=1}^N \lambda_i e_i-\sum_{i=N+1}^{2N} \lambda_i e_{i-N} \right)=
\sum_{i=1}^N \frac{\lambda_i}{N}-\sum_{i=N+1}^{2N} \frac{\lambda_i}{N}
\end{equation}
where we used the fact that $\pi_Y(e_i)=\frac{1}{N}$ and \eqref{eq:active_decomp}.
We observe that \eqref{eq:vmdec} in particular implies that $\vm$ is Lipschitz continuous.

\begin{rem}[Properties of $\vm$]\label{rem:vm}
We observe that, by construction, the function $\vm$ is positively homogeneous of degree one. According to the choice of the parameters $\mu_i^\pm$, it can be convex, concave, or more often neither.

Let us consider for example the case $N=3$, with homogeneous friction for the three contact points, namely $\mu_i^+=\mu^+$ and $\mu_i^-=\mu^-$ for $i=1,2,3$. Excluding the critical values according to \eqref{eq:uniqueness}, we have three situations. If $\mu^->2\mu^+$ then $\vm$ is convex and positive outside the origin, meaning that the crawler can move only forward. Symmetrically, if $\mu^+>2\mu^-$ then $\vm$ is concave and negative outside the origin, meaning that the crawler can move only backward. In the intermediate case $\frac{1}{2}\mu^-<\mu^+<2\mu^-$, the function $\vm$ is neither convex nor concave, and assumes both positive and negative values, with the crawler able to move in both directions. In particular, the origin is a \lq\lq monkey-saddle point\rq\rq\ (namely, a saddle point with three ridges and three ravines).
Remarkably, the mathematically desirable case of concave $\vm$ is also the less meaningful physically: indeed, the crawler can only move backward whereas we want to optimize its movement forwards, so that, for any reasonable actuation cost $f_2$, the optimal strategy is trivially to stay idle and not move.
\end{rem}


\section{Application to a one-link crawler and remarks }\label{sec:crawler}
We analyse now the information provided by Theorem \ref{nec} for the model introduced in Section \ref{sec:model}.
For simplicity, we consider only the one-link crawler. In this case,  taking into account \eqref{eq:vmdec}, 
the optimal control problem reads as follows. \medbreak

\emph{Given an interval $C:=[a,b]$, $T>0$ and a smooth convex function $f\colon \R\to \R$,
	\begin{equation*}
	\text{maximize $J(x,u):= \int_0^T \left( \frac{1}{2}\,| u(t)-\dot{x}(t)| - f(u(t)) \right)\dd t$}
	\end{equation*}
	subject to
	\begin{equation}\label{onecrawl}
	\dot{x}\in -\NN_C(x) + u \;\text{ a.e. on $[0,T]$, }\; x(T)=x(0)\in C,
	\end{equation}
	where $u(t)\in [-1,1]$ a.e. and $\int_0^T u(t)\dd t =0$.}

\medbreak
Let $(\bar{x},\bar{u})$ be an optimal trajectory-control pair.
Applying Theorem \ref{nec}, we obtain the following necessary conditions:

\medbreak
\emph{there exist $\lambda \ge 0$, a $BV$ function $p\colon [0,T]\to \mathbb{R}$, two finite positive Radon measures $\dd \xi^1$
	and $\dd \xi^2$, $\omega\in \mathbb{R}$, and $\psi\in L^1(0,T)$ such that:
	\begin{itemize}[]
		\item[1)] $\text{supp}(\dd \xi^1) \subseteq \{t\in [0,T]: \bar{x}(t)=a\}$;
		\item[2)] $\text{supp}(\dd \xi^2) \subseteq \{t\in [0,T]: \bar{x}(t)=b\}$;
		\item[3)] $\dd p = -\dd \xi^1 + \dd \xi^2\;\text{ and }\; p(T)=p(0)$;
		\item[4)] $\psi(t)= -p(t)-\omega -\lambda D_u f(\bar{u}(t))\in \NN_{[-1,1]}(\bar{u}(t))\;\text{ a.e.}$;
		\item[5)] $\lambda + \| p\|_\infty =1.$
\end{itemize}}

\medbreak
Observe first that there is a degenerate case, namely $p(t)\equiv-\omega\neq 0$ and $\lambda=0$,
that is satisfied by all trajectories of \eqref{onecrawl}, with $\dd \xi^1=\dd \xi^2=0$ and $\psi\equiv 0$, for any cost $f$.

Now we analyse a few nondegenerate cases, in the (desirable) event they occur. To simplify the analysis, we take either
$f(u)\equiv 0$ or $f(u)=\frac12 u^2$.

Let us first focus on the trajectories $x(\cdot)$ in the interior of $C$, namely the ones such that $a< x(t) < b$ for all $t$.
If $f\equiv 0$, we observe that they are all local extrema, because in this
case $\dot{x}(t)= u(t)$ a.e. and the functional $J$ vanishes in a neighbourhood of $(x,u)$.

In the case $f(u)=\frac12 u^2$, instead,
then necessary conditions provide more information. Indeed, assume again that $a<x(t)<b$ for all $t$. 
Then, by the support conditions 1) and 2) and the adjoint equation 3), $p$ is constant, so that $p+\omega$ is constant as well. 
Assume now the nondegeneracy condition $\lambda =1$ is valid. Then the extremality condition 4) reads as
\begin{equation*}
0\in p+\omega + u(t) + \NN_{[-1,1]}(u(t))\;\text{for a.e. }\; t.
\end{equation*}
The right hand side of the above expression is  a strictly monotone function of $u$, thus there exists one and only one $u^\ast$ such that
$ p+\omega + u^\ast + \NN_{[-1,1]}(u^\ast)=0$, i.e., $u(t)\equiv u^\ast$. Since all feasible controls must have zero mean,
$u\equiv 0$. Therefore, in this first nondegenerate case, the only extremal solutions that lie in the interior of $[a,b]$
for all $t$ are constant. Observe that the above analysis remains valid if $f(\cdot)$ is strictly convex with minimum
at $0$. Observe also that, still assuming $\lambda = 1$, the above argument implies that $\bar{u}$ is constant (not
necessarily zero) along any interval $I$ in which $\bar{x}$ lies in the interior of $[a,b]$.

Let us now consider the trajectories that touch the boundary of $[a,b]$. In particular, let us notice that, in order to achieve 
\lq\lq true\rq\rq\ locomotion both the boundary points must be touched by the trajectory, since one contact point can be moved 
only if $x(t)=a$ and the other only if $x(t)=b$: a necessary and sufficient condition for this to happen is $T> 2 (b-a)$. We assume again $f\equiv 0$.
In this case, the adjoint vector $p$ may not be constant, being however constant in every interval where $\bar{x}$ 
lies in the interior of $[a,b]$.
To proceed, we assume the further nondegeneracy condition 
\begin{equation}\label{nndg}
p(t)+\omega \neq 0\;\text{  for all }\; t\in [0,T].
\end{equation}
Under this condition, 4) implies that $\bar{u}(t)=-\text{sign}(p(t)+\omega)$ a.e. 
Let $\mathcal{I}:= \{ t : a< \bar{x}(t) < b\}$. Since $p$ is constant in every connected component of $\mathcal{I}$, 
$\bar{u}\in\{\pm 1\}$ is constant as well in any such component. Therefore, $\mathcal{I}$ is a finite union of open intervals, 
each of them having length $b-a$, except possibly the first and the last one, whose lengths however sum up to $b-a$ as well
due to the periodicity condition on $\bar{x}$. Summarizing, in this case optimal controls are bang-bang with finitely many switchings.

\subsection{Remarks on an assumption of Theorem \ref{th:ex}}\label{rem:ex}
The above example also illustrates why it is reasonable to expect optimal controls $\bar u(t)$ to have bounded variation. 
Indeed, we show that, given a state-control pair $(x,u)$ with unbounded variation, we can always modify it to obtain a control 
pair $(\tilde x,\tilde u)$ with $\norm{x-\tilde x}_{L^\infty[0,T]}$ arbitrarily small, and such that 
$J(\tilde x,\tilde  u)=J(x,u)$ if $f_2=0$, while $J(\tilde x,\tilde u)>J(x,u)$ if $f_2=\frac{1}{2}u^2$.

Indeed, let us consider a control $u(t)$ with unbounded variation, and let $t^*$ be a time such that in every 
neighbourhood of $t^*$ the function $u(t)$ has unbounded variation. We distinguish two cases.

Firstly, consider the case $u(t^*)\in \inter C$ and take a sufficiently small interval $[t_1,t_2]$ such that 
$t_1<t^*<t_2$ and  $u(t)\in \inter C$ for every $t\in[t_1,t_2]$. Then we define a new state-control pair as
\begin{equation*}
    (\hat x,\hat u):=\begin{cases}
    \left(\frac{(t_2-t)x(t_1)+(t-t_1)x(t_2)}{t_2-t_1},\strokedint_{t_1}^{t_2}u(s)\dd s\right) &\text{for $t\in[t_1,t_2]$}\\
    (x(t),u(t)) &\text{elsewhere}
    \end{cases}
\end{equation*}
We observe that $(\hat x,\hat u)$ satisfies \eqref{onecrawl}, with $J(\hat x,\hat u)=J(x,u)$ if $f_2=0$, and $J(\hat x,\hat u)>J(x,u)$ 
if $f_2=\frac{1}{2}u^2$. Moreover $(\hat x,\hat u)$ has bounded variation in $[t_1,t_2]$.

Secondly, we consider the case when $u(t^*)$ lies in the boundary of $C$; for simplicity we discuss the case  $u(t^*)=b$. 
We take a sufficiently small interval $[t_1,t_4]$ such that $t_1<t^*<t_4$ and  $u(t)\in (a,b]$ for every $t\in[t_1,t_4]$. 
Moreover we set $t_2:=\min \{t\in [t_1,t_4] : x(t)=b \}$ and $t_3:=\max \{t\in [t_1,t_4] : x(t)=b \}$ Then we define a new state-control pair as
\begin{equation*}
(\hat x,\hat u):=\begin{cases}
\left(\frac{(t_2-t)x(t_1)+(t-t_1)b}{t_2-t_1},\strokedint_{t_1}^{t_2}u(s)\dd s\right) &\text{for $t\in[t_1,t_2]$}\\[1mm]
\left(b, \strokedint_{t_2}^{t_3}u(s)\mathds{1}_{u^{-1}(b)}(s)\dd s\right)&\text{for $t\in[t_2,t_3]$}\\[1mm]
\left(\frac{(t_4-t)b+(t-t_3)x(t_4)}{t_4-t_3},\strokedint_{t_3}^{t_4}u(s)\dd s\right) &\text{for $t\in[t_3,t_4]$}\\[1mm]
(x(t),u(t)) &\text{elsewhere}
\end{cases}
\end{equation*}
Also in this case it is easy to see that $(\hat x,\hat u)$ satisfies \eqref{onecrawl}, with $J(\hat x,\hat u)=J(x,u)$ if $f_2=0$, 
and $J(\hat x,\hat u)>J(x,u)$ if $f_2=\frac{1}{2}u^2$. Moreover $\hat{u}$ has bounded variation in $[t_1,t_4]$.

We notice that, by the compactness of $[0,T]$, with a finite number of such modifications, we can obtain the control 
pair $(\tilde x, \tilde u)$ as desired. Furthermore, by considering sufficiently small intervals, due to the Lipschitz 
continuity of $x(t)$, we can obtain an arbitrarily small $\norm{x-\tilde x}_{L^\infty[0,T]}$.

\subsection{Remarks on the necessary conditions of Theorem \ref{nec}}\label{concl}\ \smallskip
\begin{enumerate}[wide,  labelindent=0pt, label=\arabic*.]
\item Observe that, differently from classical state constrained Bolza problems, the Hamiltonian contains only one summand of the
integral cost: the part involving $f_1 \big(g(t,x(t),u(t))-\dot{x}(t)\big)$ is missing, due to a cancellation that occurs in the proof of 
Theorem \ref{Theorem6.2}.\smallskip
\item  It is well known (see, e.g., \cite[Sec.~10.6]{vinter}  and \cite{AK}) that necessary optimality conditions for state constrained
control problems may be satisfied by \textit{all} state-control pairs.
The zero mean condition on the control $u$ provides the further multiplier $\omega\in\mathbb{R}^d$, and this is why the stronger nontriviality
condition \eqref{nndg} plays a role. However, up to now there are no sufficient conditions for \eqref{nndg} to hold. Similarly,
conditions ensuring $\lambda =1$ need to be studied, since classical results of this type do not apply to our setting.
\end{enumerate}

\section{Proof of Theorem \ref{th:ex}}\label{sec:proofex}
The proof of the existence result is based on a strong convergence argument that is essentially contained in \cite{BKS} (see also
Sec.~1.3 in \cite{Kre}). We present here a version of this result that is fit for our setting.
\begin{lem}\label{strong}
Let $C\subset X$ be closed and convex and let $g$ be as in the statement of Theorem \ref{th:ex}. Let $u_\ell,u\in L^2(0,T;\mathbb{R}^d)$,
$\ell\in\mathbb{N}$ be such that
\[
u_\ell\to u\;\text{ in }L^2. 
\]
Let $x_\ell^0\in C$ be such that
\[
x_\ell^0\to x^0 
\]
and let $x_\ell\colon [0,T]\to X$ be a solution of the Cauchy problem
\[
\begin{cases}
\dot{x}&\in - \NN_C(x) + g(t,x,u_\ell) \\
x(0)&=x_\ell^0.
\end{cases}
\]
Then there exist a solution $x$ of the Cauchy problem 
\begin{equation}\label{CP}
\begin{cases}
\dot{x}&\in - \NN_C(x) + g(t,x,u) \\
x(0)&=x^0,
\end{cases} 
\end{equation}
and a subsequence $\{x_{\ell_k}\}$ such that
\[
x_{\ell_k}\to x\;\text{ strongly in }W^{1,2}([0,T];X). 
\]
\end{lem}
\begin{proof}
It is well known (see, e.g., \cite{MaThi}) that $\{ x_\ell\}$ is uniformly bounded in $W^{1,2}([0,T];X)$, so that, up to a
subsequence, $x_\ell$ converges weakly in $W^{1,2}([0,T];X)$ to some $x\colon [0,T]\to X$. By standard arguments (see, e.g., \cite{MaThi}), $x$ is a 
solution of \eqref{CP}. Moreover,
\begin{equation}\label{gm}
g(t,x_\ell,u_\ell)\to  g(t,x,u)\;\text{ in } L^2.
\end{equation}
Set $\xi_\ell:= g(t,x_\ell,u_\ell)-\dot{x}_\ell$ and observe that $\xi_\ell(t)\in \NN_C(x_\ell(t))$ a.e. Therefore, for a.e. $t$, all $\ell$ and all $h>0$ small enough
one has both
\[
\Big\langle \xi_\ell(t),\frac{x_\ell(t+h)-x_\ell(t)}{h}\Big\rangle  \le 0
\]
and
\[
\Big\langle \xi_\ell(t),\frac{x_\ell(t-h)-x_\ell(t)}{h}\Big\rangle \le 0.
\]
By passing to the limit as $h\to 0^+$, one obtains that
\begin{equation}\label{zerom}
\langle\xi_\ell(t),\dot{x}(t)\rangle =0\quad\text{ a.e.} 
\end{equation}
The sequence $\{\xi_\ell\}$ converges weakly in $L^2$ to $\xi:= g(t,x,u)-\dot{x}$. Since $\xi(t)\in \NN_C(x(t)$ for a.e. $t$, the same argument
as above yields
\begin{equation}\label{zeroinfty}
\langle \xi(t),\dot{x}(t)\rangle = 0\quad\text{ a.e.} 
\end{equation}
By \eqref{gm}, $\dot{x}_\ell+\xi_\ell\to\dot{x}+\xi$ strongly in $L^2([0,T;X)$. Since $\dot{x}_\ell-\xi_\ell$ converges to $\dot{x}-\xi$ weakly in $L^2$,
the strong convergence is equivalent to
\begin{equation}\label{strconv}
\| \dot{x}_\ell-\xi_\ell\|_{L^2}\to \|\dot{x}-\xi\|_{L^2}. 
\end{equation}
To show \eqref{strconv}, observe that, by \eqref{zerom} and \eqref{zeroinfty},
\[
\|\dot{x}_\ell-\xi_\ell\|^2_{L^2}= \| \dot{x}_\ell\|^2_{L^2}+\|\xi_\ell\|^2_{L^2}= \|\dot{x}_\ell+\xi_\ell\|^2_{L^2}
   \to  \| \dot{x}+\xi\|^2_{L^2}= \| \dot{x}-\xi\|^2_{L^2},
\]
and the proof is concluded.
\end{proof}
Theorem \ref{th:ex} now follows easily.
\begin{proof}[Proof of Theorem \ref{th:ex}]
Let $\{(x_\ell,u_\ell)\}$ be a maximizing sequence for $\mathscr{I}$ among solutions of \eqref{traj} with $u_\ell\in\mathscr{U}_K$ for all $\ell$. By Helly's
theorem, there exists $\bar{u}$ such that, up to a subsequence, $u_\ell\to\bar{u}$ pointwise, so that $\bar{u}(t)\in U$ for all $t$,
and actually $u_\ell\to\bar{u}$ in $L^2$. Moreover, $\int_0^T\bar{u}(t)\dd t=0$. Since $C$ is compact, up to taking another subsequence we may assume that
$x_\ell(0)\to x^0$ for some $x^0\in C$. Then, by Lemma \ref{strong}, $x_\ell\to \bar{x}$ strongly in $W^{1,2}$, where $\bar{x}$ is the solution of \eqref{traj}
corresponding to $\bar{u}$ and with initial condition $\bar{x}(0)=x^0$ ($\bar{x}$ is easily seen to be $T$-periodic). Then
\[
\mathscr{I}(\bar{x},\bar{u})\ge \lim_{\ell\to\infty} \mathscr{I}(x_\ell,u_\ell), 
\]
so that $(\bar{x},\bar{u})$ is obviously an optimal state-control pair.
\end{proof}

\section{Discrete approximations of trajectories}\label{sec:discappr}
In order to apply the discretization approach, we need first to establish a result on discrete approximations of 
solutions of the general sweeping process
\begin{equation}\label{traj00}
\dot{x}\in - \NN_C(x) + g(t,x,u)\;\text{ a.e.},\;  u(\cdot) \text{ measurable with } u(t)\in U,
\end{equation}
where the polyhedron $C\subset X$, the function $g$ and the control set $U$ satisfy the assumption of Section \ref{sec:statement}.
A similar result was obtained in \cite{gmd}; here the assumptions on the reference process $(\bar{x},\bar{u})$ are weakened, as $\dot{\bar{x}}$
and $\bar{u}$ are no longer supposed to have bounded variation.

Let us first state a lemma on piece-wise constant approximations of functions with  prescribed average.
\begin{lem}\label{Lemma1}
	Let $T>0$ and $\bar{u}\in L^1(0,T;\mathbb{R}^n)$. For each $m\in\mathbb{N}$, set $t_m^i=\frac{iT}{2^m}$, $i=0,1,\ldots,2^m$, 
	and $I_m^i=[t_m^i,t_m^{i+1})$.
	Define
	\begin{equation}\label{u_m}
	u_m(t):= \sum_{i=0}^{2^m-1} \strokedint_{I_m^i} \bar{u}(s)\dd s \:\mathds{1}_{I_m^i}(t) .
	\end{equation}
	Then $u_m\to \bar{u}$ a.e. on $[0,T]$ and $\strokedint_0^T u_m=\strokedint_0^T\bar{u}$.
\end{lem}
\begin{proof}

	Set $h_m=\frac{T}{2^m}$. Let $t\in [0,T)$ and let $I_m(t):=[\tau_m(t),\tau_m(t)+h_m)$ be the unique interval 
	$I_m^i$ such that $t\in I_m^i$. By Lebesgue differentiation theorem, $\lim_{h\to 0^+} \frac{1}{h}\int_t^{t+h}\bar{u}(s)\,\dd s=
	\lim_{h\to 0^+} \frac{1}{h}\int_{t-h}^t\bar{u}(s)\dd s=\bar{u}(t)$ almost everywhere. Thus, for almost every $t$,
	\[
	\begin{split} 
	\lim_{m\to\infty}\strokedint_{I_m(t)}\bar{u}&=\lim_{m\to\infty}\frac{1}{h_m}\left(\int_{\tau_m(t)}^t\bar{u} + \int_t^{\tau_m(t)+h_m}\bar{u} \right)\\
        &=\lim_{m\to\infty}\frac{1}{h_m}\left[ (t-\tau_m(t))\Big( \bar{u}(t) + \frac{1}{t-\tau_m(t)}\int_{\tau_m(t)}^t \bar{u}(s) \dd s 
	  -\bar{u}(t)\Big)\right.\\
	&\left. \qquad\qquad +(\tau_m(t)+h_m-t)\Big( \bar{u}(t) + \frac{1}{\tau_m(t)+h_m-t}\int_t^{\tau_m(t)+h_m} \bar{u}(s) \dd s 
	  -\bar{u}(t)\Big)\right]\\
	&=\bar{u}(t) + \lim_{m\to\infty}\left(\frac{t-\tau_m(t)}{h_m}\, o(1) + \frac{\tau_m(t)+h_m -t}{h_m}\, o(1)  \right)\\
	&=\bar{u}(t).
	\end{split}
	\]
	Finally, we observe that the average of $\bar{u}$ is preserved by the discretization.
\end{proof}

Recalling that, as it is well known, all solutions of \eqref{traj00} are Lipschitz with the common Lipschitz constant $L$, 
we now state the main result of the section.
\begin{thm}\label{Theorem1}
Under the assumptions of Section \ref{sec:statement}, let
$(\bar{x},\bar{u})$ be a process for \eqref{traj}. Let $m\in\mathbb{N}$ and let $I_m^i$, $i=0,\ldots ,2^m-1$ be as in Lemma \ref{Lemma1}.
Set also $h_m=\frac{T}{2^m}$. Then there exist sequences 
\begin{equation*}
\{c_m^{ji}\}_{j=1,\ldots,\sigma,\, i=0,\ldots,2^m-1}\subset \mathbb{R},\qquad x_m\colon [0,T]\to X,\qquad r_m\colon [0,T]\to[0,+\infty) \qquad (m\in\mathbb{N}),
\end{equation*}
with the following properties:
\begin{itemize}
\item[a)] $\max_{i=0,\ldots ,2^m-1} | c_m^{ji}-c^j|\le\frac{LT}{2^m}$;  
\item[b)] $x_m(\cdot)$ is continuous and is affine on each $I_m^i$, $i=0,\ldots,2^m-1$, $x_m(0)=\bar{x}(0)$, $x_m(T)=\bar{x}(T)$,
and, for $i=0,\ldots ,2^m-1$,
\begin{equation*}
\begin{split}
\frac{x_m(t_m^{i+1})-x_m(t_m^i)}{h_m} &\in - \strokedint_{I_m^i} \big(g(s,\bar{x}(s),\bar{u}(s))-\dot{\bar{x}}(s)\big)\dd s + g(t_m^i,x_m(t_m^i),u_m(t_m^i))
+ \strokedint_{I_m^i} r_m(s)\dd s\: B_X\\
&\subset -\NN_{C_m^i} (x_m(t_m^i)) + g(t_m^i,x_m(t_m^i),u_m(t_m^i)) + \strokedint_{I_m^i} r_m(s)\dd s\: B_X,
\end{split}
\end{equation*}
where $u_m$ is defined as in \eqref{u_m} and the polyhedra $C_m^i$ will be defined in \eqref{defCm} below;
\item[c)] the sequence of polyhedral valued maps
\begin{equation*}
C_m(t):=\sum_{i=0}^{2^m-1} C_m^{i} \mathds{1}_{I_m^i}(t)
\end{equation*}
converges to $C$ with respect to the Hausdorff metric, uniformly w.r.t.~$t$, 
and $C_m^i$ satisfies PLICQ for all $i=0,\ldots,\sigma$, provided $m$ is large enough;
\item[d)] $r_m\to 0$ a.e.~on $[0,T]$ and $\|r_m\|_\infty$ is bounded uniformly w.r.t.~$m$;
\item[e)] $x_m\to \bar{x}$ strongly in $W^{1,2}([0,T];X)$.
\item[f)] $J(x_m,u_m)\to J(\bar{x},\bar{u})$;
\item[g)] $\int_0^T u_m(t)\dd t =0.$
\end{itemize}
\end{thm}
\begin{proof}
Fix $m\in \mathbb{N}$ and define, for $i=0,\ldots ,2^m-1$ and $t\in I_m^i$,
\begin{equation*}
\begin{split}
\omega_m^i &:= \frac{\bar{x}(t_m^{i+1})-\bar{x}(t_m^i)}{h_m}\\
x_m(t)&:= \bar{x}(t_m^i)+(t-t_m^i)\omega_m^i\\
&=\bar{x}(t_m^i)+(t-t_m^i)\strokedint_{I_m^i} \dot{\bar{x}}(s)\dd s.
\end{split}
\end{equation*}
Observe that, for each $i-0,\ldots ,2^m$, $x_m(t_m^i)=\bar{x}(t_m^i)\in C$, and, by Lemma \ref{Lemma1}, $x_m\to\bar{x}$ strongly in $W^{1,2}([0,T];X)$. 
Define 
\begin{equation*}
\omega_m(t) := \sum_{i=0}^{2^m-1} \omega_m^i \mathds{1}_{I_m^i}(t),\quad t\in [0,T].
\end{equation*}
Fix $i=0,\ldots , 2^m-1$ and set, for $j=1,\ldots \sigma$,
\begin{equation*}
c_m^{ji}:=
\begin{cases}
c^j & \text{ if $\langle \bar{x}(t),x_\ast^j\rangle < c^j$ for all $t\in I_m^i$}\\
\langle \bar{x} (t_m^i),x_\ast^j\rangle & \text{ otherwise,} 
\end{cases}
\end{equation*}
and
\begin{equation*}
c_m^j(t) = \sum_{i=0}^{2^m-1} c_m^{ji} \mathds{1}_{I_m^i} (t). 
\end{equation*}
We claim that $c_m^j(t)\to c^j$ uniformly on $[0,T]$ as $m\to\infty$. Indeed, set for $t\in [0,T]$
\begin{equation*}
\tau_m (t):=\max\{ t_m^i : i=0,\ldots ,2^m-1,t_m^i\le t\}\:\text{ and }\: I_m(t) = [\tau_m(t),\tau_m(t) + h_m),
\end{equation*}
and fix $j=1,\ldots ,\sigma$ and $\bar{t}\in [0,T]$. If $\langle \bar{x}(\bar{t}),x_\ast^j\rangle < c^j$, then
eventually $\langle \bar{x}(t),x_\ast^j\rangle < c^j$ for all $t\in [\tau_m (\bar{t}),\tau_m(\bar{t}) + h_m)$,
so that, eventually, $c_m^j(\bar{t}) = c^j$. Let now $\langle \bar{x}(\bar{t}),x_\ast^j\rangle = c^j$. Then
$c_m^j(\bar{t})$, that is equal to $c_m^j(\tau_m(\bar{t}))$, satisfies the conditions
\begin{equation*}
\langle \bar{x}(\tau_m(\bar{t})),x_\ast^j\rangle = c_m^j (\tau_m(\bar{t}))\le \langle \bar{x}(\bar{t}),x_\ast^j\rangle = c^j, 
\end{equation*}
where both the equality and the inequality follow from our definition of $c_m^j(t)$. Then
\begin{equation*}
\big| c_j - c_m^j (\tau_m(\bar{t}))  \big| = \big| \langle \bar{x}(\bar{t})-\bar{x}(\tau_m),x_\ast^j\rangle  \big| \le Lh_m.
\end{equation*}
This in turn implies a).

Set now
\begin{equation}\label{defzm}
g_m(t):=g(\tau_m(t),\bar{x}(\tau_m(t)),u_m(t))\quad\text{ and }\quad \zeta_m(t) = g_m(t)-\omega_m(t),\quad t\in [0,T]. 
\end{equation}
We recall that, for a.e.~$t\in [0,T]$
\begin{equation*}
g(t,\bar{x}(t),\bar{u}(t))-\dot{\bar{x}}(t) = \sum_{j=1}^\sigma \lambda_j (t) x_\ast^j 
\end{equation*}
for suitable measurable $\lambda_j(\cdot)$. Moreover we can take $\lambda_j(\cdot)\ge 0$, and such that $\lambda_j(t)=0$ for each $t$ for which 
$\langle \bar{x}(t),x_\ast^j\rangle < c^j$, namely $j\notin\actcon(\bar x(t))$. Observe that
\begin{equation}\label{++}
\begin{split}
\zeta_m(t) & = \strokedint_{[\tau_m(t),\tau_m(t)+h_m)} \big(g(s,\bar{x}(s),\bar{u}(s))-\dot{\bar{x}}(s)\big)\dd s - 
                  \strokedint_{[\tau_m(t),\tau_m(t)+h_m)} \big(g(s,\bar{x}(s),\bar{u}(s))-g_m(s)\big)\dd s\\
&= \sum_{j=1}^\sigma x_\ast^j\strokedint_{[\tau_m(t),\tau_m(t)+h_m)} \lambda_j(s)\dd s  - 
                  \strokedint_{[\tau_m(t),\tau_m(t)+h_m)} \big(g(s,\bar{x}(s),\bar{u}(s))-g_m(s)\big)\dd s.
\end{split}
\end{equation}
Define now, for $j=1,\ldots,\sigma$, $i=0,\ldots,2^m-1$,
\begin{equation}\label{defCm}
C_m^{ji}:=\left\{ x\in X : \langle x,x_\ast^j\rangle \le c_m^{ji} \right\},\quad C_m^i=\bigcap_{j=1}^\sigma C_m^{ji},
\quad C_m^j(t):=\sum_{i=0}^{2^m-1} C_m^{ji} \mathds{1}_{I_m^i}(t)
\end{equation}
and observe that, by our construction,
\begin{equation*}
\sum_{j=1}^\sigma x_\ast^j\strokedint_{I_m(t)} \lambda_j(s)\dd s  \in \NN_{C_m^j(\tau_m(t))} (x_m(\tau_m(t)) = \NN_{C_m^j(\tau_m(t))} (x_m(t))
\end{equation*}
for a.e.~$t$. Moreover, for all $t\in [0,T]$,
\begin{multline*}
\strokedint_{I_m(t)} \big(g(s,\bar{x}(s),\bar{u}(s))-g_m(s)\big)\dd s = 
      \strokedint_{I_m(t)} \big(g(s,\bar{x}(s),\bar{u}(s))-g(\tau_m(s),\bar{x}(\tau_m(t)),\bar{u}(s))\big)\dd s\\
\qquad + \strokedint_{I_m(t)} \big(g(\tau_m(s),\bar{x}(\tau_m(t)),\bar{u}(s))-g_m(s)\big)\dd s,
\end{multline*}
where we have used the fact that, for $s\in I_m(t)$, $\tau_m(s)=\tau_m(t)$. 
Now, by the uniform continuity of $g$
w.r.t.~$x$ and $u$ in the domain of interest the proof of b) is concluded.

The remaining claims are immediate consequences of the construction and of Lemma \ref{Lemma1}, taking into account that $\|\dot{z}_m\|_\infty$
is bounded uniformly w.r.t.~$m$.  In particular, the Hausdorff convergence in item c) follows from the fact that active constraints
are eventually constant, while statement f) follows from the pointwise a.e. convergence of the sequences
$\dot{z}_m$ and $u_m$. Moreover, $u_m$ has zero mean by construction.
\end{proof}

\section{A discrete optimization problem}\label{sec:discopt}
Let $(\bar{x},\bar{u})$ be a local $W^{1,2}$-optimal process for problem (P), i.e., there exists 
$\bar{\varepsilon}>0$ such that for all processes $(x,u)$ of \eqref{traj} with 
$\| x-\bar{x}\|_{W^{1,2}([0,T];X)}+\|u-\bar{u}\|_{L^2([0,T];X)}<\bar{\varepsilon}$
one has $J(x,u)\le J(\bar{x},\bar{u})$. Let $\{ (x_m,u_m)\}$ be the sequence of approximations of $(\bar{x},\bar{u})$ constructed according to
Lemma \ref{Lemma1} and Theorem \ref{Theorem1} and set, for $m\in\mathbb{N}$,
\[
\alpha_m = \sum_{i=0}^{2^m-1} \int_{I_m^i}\left( \left|\frac{x_m^{i+1}-x_m^i}{h_m}-\dot{\bar{x}}(t) \right|^2+\left|u_m^i-\bar{u}(t) \right|^2 \right)\dd t. 
\]
By Theorem 5.2, $\alpha_m\to 0$ as $m\to\infty$. Set
\begin{equation}\label{menodue}
\kappa_m=\frac{1}{\sqrt{\alpha_m}}. 
\end{equation}
Consider the following family of finite dimensional optimization problems ($P_m$):

\smallskip
\noindent\textbf{Problem \textup{($P_m$)}} \emph{
Let $\bar{\varepsilon}$ be given by the definition of local $W^{1,2}$ optimal process;
let $m\in \mathbb{N}$ be given; let $I_m^i$, $i=0,\ldots,2^m-1$, and $u_m$ be as in Lemma \ref{Lemma1}; 
let $r_m(\cdot)$ and $c_m^{ji}$, $j=1,\ldots,\sigma$, $i=0,\ldots ,2^m-1$, be as in Theorem \ref{Theorem1}, 
and set $c_m^{j2^m}:=c_m^{j1}$, $j=1,\ldots,\sigma$. 
Writing $(z_m;w_m;\rho_m)=(z_m^0,\ldots,z_m^{2^m};w_m^0,\ldots ,w_m^{2^m-1};$ $\rho_m^0,\ldots,\rho_m^{2^m-1})$, we want to
\begin{equation*}
\begin{split}
\text{maximize }\; J_m(z_m,w_m)
      &:=h_m\sum_{i=0}^{2^m-1}\left( f_1 \Big(g(t_m^i,z_m^i,w_m^i)-\frac{z_m^{i+1}-z_m^i}{h_m} \Big)-f_2(t_m^i,w_m^i)\right)\\
      &\qquad\quad -\frac{\kappa_m}{2}\left(\big|z_m^0-\bar{x}(0)\big|^2+
	\sum_{i=0}^{2^m-1}\int_{I_m^i} \left( \left| \frac{z_m^{i+1}-z_m^i}{h_m}-\dot{\bar{x}}(t) \right|^2 + 
| w_m^i-\bar{u}(t)|^2 \right)\dd t\right)
\end{split}
\end{equation*}
over discrete processes $(z_m,w_m,\rho_m)$ such that
\begin{align}
\tag{$m_1$}
&\langle x_\ast^j,z_m^i\rangle \le c_m^{ji}, \;i=1,\ldots,2^m, \;j=1,\ldots,\sigma;\\
\tag{$m_2$}
&\frac{z_m^{i+1}-z_m^i}{h_m}= - \strokedint_{I_m^i} \big(g(s,\bar{x}(s),\bar{u}(s))-\dot{\bar{x}}(s)\big)\dd s + g(t_m^i,z_m^i,w_m^i)\\ 
 &    \quad\qquad \qquad \qquad  +  \rho_m^i\, \strokedint_{I_m^i} r_m(s)\dd s, \qquad i=0,\ldots,2^m-1;\nonumber\\
\tag{$m_3$} 
& z_m^{2^m}=z_m^0;\\
\tag{$m_4$}&\big|z_m^0-\bar{x}(0)\big|^2+
 \sum_{i=0}^{2^m-1} \int_{I_m^i} \left( \left| \frac{z_m^{i+1}-z_m^i}{h_m}-\dot{\bar{x}}(t) \right|^2 + | w_m^i-\bar{u}(t)|^2 \right)\dd t\le
\frac{\bar{\varepsilon}}{2};\\
\tag{$m_5$} 
&\sum_{i=0}^{2^m-1} w_m^i =0;\\
\tag{$m_6$} 
& w_m^i\in U,\; i=0,\ldots,2^m-1;\\
\tag{$m_7$}
&\rho_m^i\in B_X,\; i=0,\ldots,2^m-1.
\end{align}}
By standard finite dimensional programming arguments, problem ($P_m$) admits optimal processes.

The following result is pivotal in the method of discrete approximations. It is similar to, e.g., Theorem 4.3 in \cite{gmd}, 
with a difference: we do not assume relaxation stability. Indeed our method allows to treat a non-concave integral functional, as the functional that
appears in Sections 3.4 and 4, without passing through the relaxed problem.
\begin{thm}\label{Theorem4.2}
Let  the assumptions of Theorem \ref{nec} hold. Let $(\bar{x},\bar{u})$ be a $W^{1,2}$-optimal process for Problem ($P$) 
and let $(\bar{x}_m,\bar{u}_m)$ be optimal processes for Problems ($P_m$).
With an abuse of notation, consider $\bar{x}_m$, resp. $\bar{u}_m$, as piecewise affinely, resp. piecewise constantly, extended to the whole of $[0,T]$.
Then
\begin{equation*}
\begin{split}
&\bar{x}_m \to \bar{x}\;\text{ strongly in }\; W^{1,2}([0,T];X)\\
&\bar{u}_m \to \bar{u}\;\text{ strongly in }\; L^2([0,T];\mathbb{R}^d).
\end{split}
\end{equation*}
More precisely,
\begin{equation}\label{menouno}
\lim_{m\to\infty} \kappa_m \left(\big|\bar{x}_m^0-\bar{x}(0)\big|^2+
	\sum_{i=0}^{2^m-1}\int_{I_m^i} \left( \left| \dot{\bar{x}}_m(t)-\dot{\bar{x}}(t) \right|^2 + 
| \bar{u}_m(t)-\bar{u}(t)|^2 \right)\dd t\right)=0.
\end{equation}
Consequently, $J_m(\bar{x}_m,\bar{u}_m)\to J(\bar{x},\bar{u})$.
\end{thm}
\begin{proof}
Suppose by contradiction that, possibly along a subsequence,
\begin{equation}\label{zero}
\lim_{m\to\infty} \kappa_m \left(\big|\bar{x}_m^0-\bar{x}(0)\big|^2+
	\sum_{i=0}^{2^m-1}\int_{I_m^i} \left( \left| \dot{\bar{x}}_m(t)-\dot{\bar{x}}(t) \right|^2 + 
| \bar{u}_m(t)-\bar{u}(t)|^2 \right)\dd t\right)=\gamma\in (0,+\infty]. 
\end{equation}
Thanks to ($m_4$), the sequence $\{ ( \bar{x}_m,\bar{u}_m)\}$ is bounded in $W^{1,2}([0,T];X)\times L^2(0,T;\mathbb{R}^d)$, so that,
up to taking a subsequence, it converges to some $(\tilde{x},\tilde{u})$ weakly in the same space.

By our definition of $\kappa_m$,
\begin{equation}\label{uno}
\lim_{m\to\infty} \kappa_m \left(\big|x_m^0-\bar{x}(0)\big|^2+
	\sum_{i=0}^{2^m-1}\int_{I_m^i} \left( \left| \dot{x}_m(t)-\dot{\bar{x}}(t) \right|^2 + 
| u_m(t)-\bar{u}(t)|^2 \right)\dd t\right)=0 ,
\end{equation}
where we recall that $(x_m,u_m)$ is the sequence of approximations of $(\bar{x},\bar{u})$ constructed in Lemma \ref{Lemma1} and in Theorem \ref{Theorem1},
hence, in particular, $x_m^0=\bar{x}(0)$. Since $(\bar{x}_m,\bar{u}_m)$ is an optimal process for ($P_m$),
\begin{equation}\label{due}
J_m(\bar{x}_m,\bar{u}_m)\ge J_m(x_m,u_m)\quad\forall m\in \mathbb{N}. 
\end{equation}
By \eqref{uno} and by strong convergence (see Theorem \ref{Theorem1}),
\begin{equation}\label{tre}
J_m(x_m,u_m)\to J(\bar{x},\bar{u})\;\text{ as $m\to\infty$.} 
\end{equation}
Moreover, by ($m_2$), ($m_4$) and our assumptions, the sequence
\[
h_m \sum_{i=0}^{2^m-1} \left( f_1 \Big(g(t_m^i,\bar{x}_m^i,\bar{u}_m^i)-\frac{\bar{x}_m^{i+1}-\bar{x}_m^i}{h_m} \Big)-f_2(t_m^i,\bar{u}_m^i)\right)
\]
is uniformly bounded, so that the sequence
\[
J_m(\bar{x}_m,\bar{u}_m)+\frac{\kappa_m}{2}\left(|\bar{x}_m(0)-\bar{x}(0)|^2+\int_0^T \left( \left|\dot{\bar{x}}_m(t)-\dot{\bar{x}}(t) \right|^2
+\left| \bar{u}_m(t)-\bar{u}(t)\right|^2 \right)\dd t \right)
\]
is uniformly bounded from above. Moreover, thanks to \eqref{due} and \eqref{tre}, the same sequence is also uniformly bounded from below.
This implies in turn that the sequence
\[
\kappa_m \left(\big|\bar{x}_m^0-\bar{x}(0)\big|^2+
	\sum_{i=0}^{2^m-1}\int_{I_m^i} \left( \left| \dot{\bar{x}}_m(t)-\dot{\bar{x}}(t) \right|^2 + 
| \bar{u}_m(t)-\bar{u}(t)|^2 \right)\dd t\right)
\]
is uniformly bounded, so that, in particular, $\gamma < +\infty$.
As a consequence, $(\tilde{x},\tilde{u})=(\bar{x},\bar{u})$, and the convergence $(\bar{x}_m,\bar{u}_m)\to (\bar{x},\bar{u})$ is indeed strong
in $W^{1,2}([0,T];X)\times L^2(0,T;\mathbb{R}^d)$. Thanks to this fact, 
\[
\lim_{m\to\infty} J_m(\bar{x}_m,\bar{u}_m) = J(\bar{x},\bar{u}). 
\]
Therefore, \eqref{zero} implies that
\[
J(\bar{x},\bar{u})-\frac{\gamma}{2} \ge J(\bar{x},\bar{u}), 
\]
and this contradiction completes the proof.

\end{proof}

\section{Necessary conditions for the discrete approximate problem}\label{sec:discnec}
Throughout this section, the assumptions of Theorem \ref{nec} are supposed to hold.

\medbreak

Fix $m\ge 1$. In order to proceed with deriving necessary optimality conditions for problems ($P_m$), we will introduce some further notations.
We set 
\begin{equation*}
\mathscr{X} := (x_m^0,\ldots,x_m^{2^m-1};w_m^0,\ldots,w_m^{2^m-1}; \rho_m^0,\ldots,\rho_m^{2^m-1};\Delta_m^0,\ldots,\Delta_m^{2^m-1})\in 
X^{2^m}\times \mathbb{R}^{2^m d}\times X^{2^m}\times X^{2^m}.
\end{equation*}
With the understanding that $\Delta_m^i=\frac{x_m^{i+1}-x_m^i}{h_m}$, we will write the functional $J_m$ as depending on $\mathscr{X}$.

We define furthermore the following 
maps, for $i=0,\ldots,2^m-1$: for $x\in X$, $w\in U$, and $\rho\in B_X$, we set
\begin{equation*}
\Gamma_m^i(x,w,\rho):=  -\strokedint_{I_m^i}\big(g(s,\bar{x}(s),\bar{u}(s))-\dot{\bar{x}}(s)\big)\dd s + g(t_m^i,x,w) +
\rho \strokedint_{I_m^i}r_m(s)\dd s \;\;(\in X).
\end{equation*}
The computation of necessary conditions for ($P_m$) will be carried out in two steps. In the first step we will show how to set ($P_m$) in the
framework of classical results on finite dimensional optimization, while in the second one the calculations for this particular case will be
carried out.
\begin{thm}\label{Theorem6.1}
Let $\bar{\mathscr{X}}=(\bar{x}_m,\bar{w}_m,\bar{\rho}_m,\bar{\Delta}_m)$ be an optimal process for ($P_m$). 
Then there exist $\lambda_m >0$, $\omega_m \in \mathbb{R}^d$, $\psi_m^i\in \mathbb{R}^d$, $\beta_m^i,\, \eta_m^i\in X$,
$\xi_m^i= (\xi_m^{i1},\ldots,\xi_m^{i\sigma})\in \mathbb{R}_+^\sigma$, $p_m^i \in X$, and
$\mathscr{X}_m^\ast\in X^{2^m}\times \mathbb{R}^{2^md}\times X^{2^m}\times X^{2^m}$, $i=0,\ldots,2^m$ such that

\begin{equation}\label{necd1}
\xi_m^{ij} \,\big(\langle x_\ast^j,\bar{x}_m^i\rangle - c_m^{ij}\big)=0,\; i=0,\ldots ,2^{m-1},\; j=1,\ldots,\sigma, 
\end{equation}
and, for $i=0,\ldots ,2^m-1$,
\begin{multline}\label{necd2}
 \Bigg(\kappa_m^i(\bar{x}(0)-\bar{x}_m^i)- \lambda_m D_x g(t_m^i,\bar{x}_m^i,\bar{w}_m^i)^\ast\eta_m^i
    -\sum_{j=1}^\sigma \frac{\xi_m^{ij}}{h_m} x_\ast^j+\frac{p_m^{i+1}-p_m^i}{h_m},\\
    -\lambda_m \big(D_w g(t_m^i,\bar{x}_m^i,\bar{w}_m^i)^\ast\eta_m^i+\kappa_m(\bar{w}_m^i-u_m^i)+D_w f_2(t_m^i,w_m^i)\big)
 -  \frac{\omega_m}{h_m}-\frac{\psi_m^i}{h_m},\\
 -\frac{\beta_m^i}{h_m},p_m^i+\lambda_m\bigg(\eta_m^i+
    \kappa_m\Big(\frac{\bar{x}_m^{i+1}-\bar{x}_m^i}{h_m}-\strokedint_{I_m^i}\dot{\bar{x}}(t)\dd t \Big)
   \bigg)  \Bigg) \in \NN_{\mathrm{graph}\, (\Gamma_m^i)}(\bar{\mathscr{X}}),
\end{multline}
where
\begin{equation}\label{7.2.5}
\kappa_m^i=\kappa_m\;\text{ for $i=0$ and $\kappa_m^i=0$ for $i\neq 0$, where $\kappa_m$ is defined in \eqref{menodue},}
\end{equation}
\begin{equation}\label{per_p}
p_m^{2^m}=p_m^0,
\end{equation}
and
\begin{equation}\label{subgrs}
\psi_m^i\in \NN_U(\bar{w}_m^i),\quad \beta_m^i\in \NN_{B_X}(\bar{\rho}_m^i),\quad 
\eta_m^i\in\partial f_1 \left(g(t_m^i,\bar{x}_m^i,\bar{w}_m^i)-\frac{\bar{x}_m^{i+1}-\bar{x}_m^i}{h_m}\right).
\end{equation}
\end{thm}
\begin{proof}
We begin by arranging in two different categories the constraints in problem ($P_m$). The variable $\mathscr{X}$ fulfils the following requirements:
\begin{align}
\label{c1} \Phi(\mathscr{X})&:=|x_m^0-\bar{x}(0)|^2+ \sum_{i=0}^{2^m-1}\int_{I_m^i} |(\Delta_m^i,w_m^i)-(\dot{\bar{x}}(t),\bar{u}(t))|^2\dd t 
   - \frac{\bar{\varepsilon}}{2}\le 0\\
\label{c2a} \delta_m^i(\mathscr{X})&:= x_m^{i+1}-x_m^i-h_m\Delta_m^i =0,\quad i=1,\ldots ,2^m-1\\
\label{c2b}\delta_m^0(\mathscr{X})&:=x_m^1-x_m^{2^m}-h_m\Delta_m^0=0\\
\label{c3} h_m^{ij}(\mathscr{X})&:=\langle x_\ast^j,x_m^i\rangle - c_m^{ij}\le 0,\quad i=1,\ldots ,2^m,\; j=1,\ldots,\sigma
\end{align}
together with
\begin{align}
\label{g2} \mathscr{X}\in\Xi_m^i&:= \left\{ \mathscr{X} : \Delta_m^i = \Gamma_m^i (x_m^i,w_m^i,\rho_m^i)\right\},\quad i=0,\ldots ,2^m-1\\
\label{g3} \mathscr{X}\in \mathcal{B}_m^i&:= \{ \mathscr{X} : |\rho_m^i|\le 1\},\quad i=0,\ldots ,2^m-1\\
\label{g4} \mathscr{X}\in\Omega_m^i&:= \{ \mathscr{X} : w_m^i\in U\},\quad i=0,\ldots ,2^m-1\\
\label{g5} \mathscr{X} \in \Omega_m&:= \{ \mathscr{X} : \sum_{i=0}^{2^m-1} w_m^i =0\}.
\end{align}
Recalling Theorem \ref{Theorem4.2}, the constraint \eqref{c1} is eventually inactive  and therefore will be neglected 
in the computations of necessary conditions. Applying classical results in mathematical programming we
obtain a set of necessary conditions for ($P_m$) that read as follows.

\medbreak
\noindent\emph{There exist $\lambda_m>0$, $\omega_m \in \mathbb{R}^d$, $\psi_m^i\in \mathbb{R}^d$,
$\xi_m^i= (\xi_m^{i1},\ldots,\xi_m^{i\sigma})\in \mathbb{R}_+^\sigma$, $p_m^i \in X$, and 
$\mathscr{X}_i^\ast\in X^{2^m}\times X^{2^m}\times X^{2^m}\times X^{2^m}$, $i=0,\ldots,2^m$ such that
\begin{equation}\label{disc1}
\mathscr{X}_i^\ast \in \NN_{\Xi_m^i} (\mathscr{\bar{X}}) + \hat{\NN}_{\mathcal{B}_m^i}(\bar{\mathscr{X}})+
            \hat{\NN}_{\Omega_m^i}(\bar{\mathscr{X}}), \quad i=0,\ldots, 2^m-1 
\end{equation}
\begin{equation}\label{disc2}
\mathscr{X}_{2^m}^\ast \in \hat{\NN}_{\Omega_m}(\bar{\mathscr{X}}),
\end{equation}
where
\begin{equation*}
\begin{split}
\hat{\NN}_{\mathcal{B}_m^i}(\mathscr{X}) &= (0,\ldots,0,\beta_m^{i},0,\ldots,0),\;
         \text{ with }\; \beta_m^{i}\in \NN_{\mathcal{B}_m^i}(\bar{\rho}_m^i)\\  
\hat{\NN}_{\Omega_m^i}(\mathscr{X})&=(0,\ldots,0,\psi_m^{i},0,\ldots,0),\;
         \text{ with }\; \psi_m^{i}\in \NN_{U}(\bar{w}_m^i)\\
\hat{\NN}_{\Omega_m}(\mathscr{X})&=(0,\ldots,0,\underline{\omega}_m,0,\ldots,0),\;
         \text{ with }\; \underline{\omega}_m=(\omega_m,\ldots,\omega_m)\in \mathbb{R}^d,
\end{split}
\end{equation*}
together with
\begin{equation}\label{discrnec}
\begin{split}
-\sum_{i=0}^{2^m} \mathscr{X}_i^\ast &\in \lambda_m \partial J_m(\bar{\mathscr{X}}) 
        +\sum_{i=1}^{2^m}\sum_{i=1}^\sigma \xi_m^{ij} \nabla h_m^{ij} (\bar{\mathscr{X}}) 
        +\sum_{i=0}^{2^m-1} \nabla \delta_m^i (\bar{\mathscr{X}})^\ast p_m^i,\\
&\quad \text{ where }\; \xi_m^{ij}\, h_m^{ij}(\bar{\mathscr{X}}) =0,\quad i=1,\ldots,2^m,\; j=1,\ldots,\sigma.
\end{split}
\end{equation}}

\medbreak
We now write componentwise the above expression, making first explicit the (sub)gradients.
Invoking the nonsmooth chain rule (see, e.g., Theorem 10.6 and Example 10.8 in \cite{RW}) we obtain
\begin{equation*}
\begin{split}
\partial J_m(\bar{\mathscr{X}})& \subset \bigg( h_m D_x g(t_m^0,\bar{x}_m^0,\bar{w}_m^0)^\ast 
         \partial f_1 \big(g(t_m^0,\bar{x}_m^0,\bar{w}_m^0)- \bar{\Delta}_m^0\big)+\kappa_m (\bar{x}(0)-\bar{x}_m^0),\\
&\quad\; \big[h_m D_x g(t_m^i,\bar{x}_m^i,\bar{w}_m^i)^\ast 
         \partial f_1 \big(g(t_m^i,\bar{x}_m^i,\bar{w}_m^i)-\bar{\Delta}_m^i\big)\big]_{i=1},\ldots ,2^m;\\
&\quad\; \Big[h_m\Big(D_w g(t_m^i,\bar{x}_m^i,\bar{w}_m^i)^\ast 
         \partial f_1 \big(g(t_m^i,\bar{x}_m^i,\bar{w}_m^i) -\bar{\Delta}_m^i\big)\\
&\qquad\qquad\qquad - D_w f_2(t_m^i,\bar{w}_m^i)+ \kappa_m\int_{I_m^i}(\bar{u}(t)-\bar{w}_m^i)\dd t\Big)\Big]_{i=0,\ldots ,2^m-1};\\
&\quad\; 0_{X^{2^m-1}}; \Big[-h_m \partial  f_1 \big(g(t_m^i,\bar{x}_m^i,\bar{w}_m^i)-\bar{\Delta}_m^i\big)+\kappa_m
       \int_{I_m^i} (\dot{\bar{x}}(t)-\bar{\Delta}_m^i)\dd t\Big]_{i=0,\ldots,2^m-1}\bigg).
\end{split}
\end{equation*}
Moreover
\begin{equation*}
\big(\nabla h_m^{ij} (\bar{\mathscr{X}})\big)_{X^i}=x_\ast^j,\quad i=1,\ldots ,2^m, j=1,\ldots ,\sigma, 
\end{equation*}
\begin{equation*}
\Big( \sum_{\ell=0}^{2^m-1}\nabla\delta_m^\ell(\bar{\mathscr{X}})^\ast p_m^\ell \Big)_{X^i}=
\begin{cases}
p_m^{i-1}-p_m^i & \; \text{ for }\; 1 \le i \le 2^m-1\\
p_m^{2^m-1}-p_m^0  & \; \text{ for }\;  i =2^m,
\end{cases}
\end{equation*}
\begin{equation*}
\Big( \sum_{\ell=0}^{2^m-1}\nabla\delta_m^\ell(\bar{\mathscr{X}})^\ast p_m^\ell \Big)_{\Delta^i}= -h_m p_m^i,\quad i=0,\ldots , 2^m-1.
\end{equation*}
Thus we obtain from \eqref{discrnec}, for a suitable $\eta_m^\ell\in\partial f_1 
(g(t_m^\ell,\bar{x}_m^\ell,\bar{w}_m^\ell)-\bar{\Delta}_m^\ell)$,
\begin{equation*}
-\sum_{i=0}^{2^m} x_i^{\ast\ell}=-x_\ell^{\ast\ell}=\lambda_m h_m D_x g(t_m^\ell,\bar{x}_m^\ell,\bar{w}_m^\ell)^\ast\eta_m^\ell + 
        \sum_{j=1}^\sigma \xi_m^{\ell j}x_\ast^j-p_m^\ell + p_m^{\ell-1},
  \quad \ell=1,\ldots ,2^m,
\end{equation*}
where we have set $p^{2^m}=p_m^0$. For $\ell=0,\ldots ,2^m-1$ we have moreover
\begin{equation*}
\begin{split}
-\sum_{i=0}^{2^m} w_i^{\ast\ell }&=-w_\ell^{\ast\ell} = \lambda_m h_m \big(D_w g(t_m^\ell,\bar{x}_m^\ell,\bar{w}_m^\ell)^\ast \eta_m^\ell 
         + u_m^\ell - \bar{w}_m^\ell + D_w f_2(t_m^\ell,\bar{w}_m^\ell) \big),
\\
-\sum_{i=0}^{2^m}\rho_i^{\ast\ell} & = -\rho_\ell^{\ast\ell} =0,\\
-\Delta_\ell^{\ast\ell} &= -h_m p_m^\ell-\lambda_m\Big( h_m(\eta_m^\ell + \bar{\Delta}_m^\ell) -\int_{I_m^i}\dot{\bar{x}}(t)\dd t  \Big).
\end{split}
\end{equation*}
Observe now that \eqref{disc1} and \eqref{disc2} can be rewritten as
\begin{equation}\label{disc3}
(x_{i+1}^{\ast i+1}, w_i^{\ast i}-\psi_m^i-\omega_m,\rho_i^{\ast i}-\beta_m^i,\Delta_i^{\ast i})
  \in \NN_{\mathrm{graph}(\Gamma^i)}(\bar{\mathscr{X}}^i),\quad i=0,\ldots ,2^m-1
\end{equation}
for suitable vectors $\psi_m^i\in \NN_U(\bar{w}^i)$ and $\beta_m^i\in \NN_{B_X}(\bar{\rho}_m^i)$.
Dividing by $h_m$ the left hand side of \eqref{disc3} and taking into account the above list of necessary conditions, one arrives to \eqref{necd1}
and \eqref{necd2}. The proof is concluded.
\end{proof}
In the next result we obtain more explicit necessary conditions by computing the normal cone in the right hand side of \eqref{necd2}.
\begin{thm}\label{Theorem6.2}
Let $\bar{\mathscr{X}}=(\bar{x}_m,\bar{w}_m,\bar{\rho}_m,\bar{\Delta}_m)$ be an optimal process for ($P_m$). 
Then there exist $\lambda_m \in\mathbb{R}$, $\omega_m \in \mathbb{R}^d$, $\psi_m^i\in \mathbb{R}^d$, $\beta_m^i,\, \eta_m^i\in X$, 
$\xi_m^i= (\xi_m^{i1},\ldots,\xi_m^{i\sigma})\in \mathbb{R}_+^\sigma$, $p_m^i \in X$, and 
$\mathscr{X}_m^\ast\in X^{2^m}\times X^{2^m}\times X^{2^m}\times X^{2^m}$, $i=0,\ldots,2^m$ such that \eqref{necd1} and \eqref{subgrs} hold, together with
\begin{equation}\label{lambda}
\lambda_m >0 
\end{equation}
and, for, $i=0,\ldots,2^m-1$,
\begin{align}
\begin{split}
\label{px} \frac{p_m^{i+1}-p_m^i}{h_m}&=-D_x g(t_m^i,\bar{x}_m^i,\bar{w}_m^i)^\ast (p_m^i-\lambda_m\vartheta_m^{x,i})
+\sum_{j\in I_m^i}\frac{\xi_m^{ij}}{h_m} x_\ast^j+\lambda_m\kappa_m^i (\bar{x}_m^0-\bar{x}(0)),
\end{split}
\end{align}
where $I_m^i=\left\{ j=1,\ldots,\sigma : \langle x_\ast^j,x_m^i\rangle = c_m^{ij}\right\}$,
     $\kappa_m^i$ is as in \eqref{7.2.5}, and we have set
\begin{equation}
\label{defthx} \vartheta_m^{x,i}:= \strokedint_{I_m^i} \dot{\bar{x}}(t)\dd t - \frac{\bar{x}_{m}^{i+1}-\bar{x}_m^i}{h_m},
\end{equation}
\begin{align}
\begin{split}
\label{pw} \frac{\psi_m^{i}}{h_m}&= \lambda_m \Big( \kappa_m^i(u_m^i-\bar{w}_m^i)-D_w f_2(t_m^i,\bar{w}_m^i)
     -D_w g(t_m^i,\bar{x}_m^i,\bar{w}_m^i)^\ast\vartheta_m^{x,i}\Big)\\
   &\qquad\qquad\qquad -   \frac{\omega_m}{h_m} + D_w g(t_m^i,\bar{x}_m^i,\bar{w}_m^i)^\ast p_m^i   \in \NN_U(\bar{w}_m^i)
\end{split}
\\
\label{pr} \frac{\beta_m^i}{h_m}&= \vartheta_m^{r,i} \left(p_m^i + \lambda_m (\eta_m^i-\vartheta_m^{x,i})\right)\in \NN_{B_X}(\bar{\rho}_m^i)
\end{align}
where we have set
\begin{align}
\label{defthr} \vartheta_m^{r,i}&:= \strokedint_{I_m^i} r_m(t)\dd t.
\end{align}
Moreover, in \eqref{px} we have
\begin{equation}\label{p2mx}
p_m^{2^m}=p_m^0. 
\end{equation}
\end{thm}
\begin{proof}
The computation of the normal cone to the graph of $\Gamma_m^i$, recalling \eqref{necd2}, yields for $i=0,\ldots,2^m-1$
\begin{equation*}
\begin{split}
\left(
\begin{array}{l}
\kappa_m^i(\bar{x}(0)-\bar{x}^0)-\lambda_m D_x g(t_m^i,\bar{x}_m^i,\bar{w}_m^i)^\ast\eta_m^i
      -\sum_{j\in I_m^i}\frac{\xi_m^{ij}}{h_m} x_\ast^j + \frac{p_m^{i+1}-p_m^i}{h_m}\\
  -\lambda_m\big( D_w g(t_m^i,\bar{x}_m^i,\bar{w}_m^i)^\ast\eta_m^i +\kappa_m(\bar{w}_m^i- u_m^i) + D_w f_2(t_m^i,\bar{w}_m^i)\big) - 
           \frac{\omega_m}{h_m} - \frac{\psi_m^i}{h_m}\\
-\frac{\beta_m^i}{h_m} 
\end{array}
\right)\\
= -\left(
\begin{array}{c}
D_x g(t_m^i,\bar{x}_m^i,\bar{w}_m^i)^\ast\\
D_w g(t_m^i,\bar{x}_m^i,\bar{w}_m^i)^\ast\\
\vartheta_m^{r,i}\mathbb{I} 
\end{array}
\right)
\begin{array}{c}
\left(p_m^i + \lambda_m (\eta_m^i - \vartheta_m^{x,i})
  \right)\\\vspace{7truept}\\
\end{array}
\end{split}
\end{equation*}
where $\mathbb{I}$ denotes the identity matrix in $X$.
By computing the above product and recalling the terminal condition from Theorem \ref{Theorem6.1}, the assertions follow.
\end{proof}
\section{Proof of Theorem \ref{nec}: passing to the limit}\label{sec:main-proof}
We conclude the proof of Theorem \ref{nec} by performing a limiting procedure along the necessary conditions for problems ($P_m$) that
were proved in Theorem \ref{Theorem6.2}.
\begin{proof}[Proof of Theorem \ref{nec}.]
Referring to the statement of Theorem \ref{Theorem6.2}, we set
\begin{itemize}
\item $p_m(t)= p_m^i + (t-t_m^i) (p_m^{i+1}-p_m^i)$, for $t\in [t_m^i,t_m^{i+1})$, $i=0,\ldots,2^m -1$
\item $p_m (T) = p_m(0)$
\item $\xi_m^j(t) = \sum_{i=0}^{2^m-1} \frac{\xi_m^{ij}}{h_m} \mathds{1}_{[t_m^i,t_m^{i+1})}(t)$, $t\in [0,T)$, $j=1,\ldots,\sigma$
\item $\psi_m(t) = \sum_{i=0}^{2^m-1} \frac{\psi_m^i}{h_m} \mathds{1}_{[t_m^i,t_m^{i+1})}(t)$, $t\in [0,T)$
\item $\eta_m(t) = \sum_{i=0}^{2^m-1} \eta_m^i \mathds{1}_{[t_m^i,t_m^{i+1})}(t)$, $t\in [0,T)$
\item $\beta_m (t) = \sum_{i=1}^{2^m-1}\frac{\beta_m^i}{h_m} \mathds{1}_{[t_m^i,t_m^{i+1})}(t)$, $t\in [0,T)$
\item $\vartheta_m(t) = \sum_{i=0}^{2^m-1}\frac{\vartheta_m^i}{h_m}\mathds{1}_{[t_m^i,t_m^{i+1})}(t)$, $t\in [0,T)$, where 
$\vartheta_m^i = (\vartheta_m^{x,i},\vartheta_m^{r,i})$
\item $\omega_m^\ast = \frac{\omega_m}{h_m}$.
\end{itemize}
Observe first that, by \eqref{defthx}, \eqref{defthr}, 
 and \eqref{menouno}, $\vartheta_m\to 0$ in $L^1(0,T;X^2)$  and $\kappa_m (\bar{x}_m^0-\bar{x}(0))\to 0$. 
In particular, there
exists $\Lambda\in\mathbb{R}$ such that, for every $m$, 
\begin{equation}\label{Lambda}
\kappa_m (\bar{x}_m^0-\bar{x}(0))+\sum_{i=0}^{2^m-1} \| \vartheta_m^i\| \le \Lambda.
\end{equation}
Since all conditions appearing in the statement of Theorem \ref{Theorem6.2} are positively homogeneous of degree one, thanks to \eqref{lambda}
we assume without loss of generality that
\begin{equation}\label{norm}
\lambda_m + |p_m(T)| + \sum_{i=0}^{2^m-1} \big| \sum_{j=1}^\sigma \xi_m^{ij} x_\ast^j\big| + \frac{|\omega_m|}{h_m} 
      + \sum_{i=1}^{2^m-1}| \psi_m^i| =1, 
\end{equation}
i.e.,
\begin{equation*}
\lambda_m +  |p_m(T)| +\big\| \sum_{j=1}^\sigma \xi_m^{j} x_\ast^j \big\|_{L^1(0,T;X)} + |\omega_m^\ast| + \| \psi_m\|_{L^1(0,T;X)}=1.
\end{equation*}
By compactness, there exists a subsequence, that we do not relabel, and there exist $\lambda\ge 0$, $\omega\in X$, 
$d\xi^j\in \mathcal{C}_+^{\ast}([0,T];X)$, $i=1,\ldots,\sigma$, such that
\begin{equation*}
\begin{split}
\lambda_m &\to \lambda\\
\omega_m^\ast&\to\omega\\
\sum_{j=1}^\sigma x_\ast^j\dd \xi_m^j&\to\sum_{j=1}^\sigma x_\ast^j\dd \xi^j\qquad\text{in }\; \mathcal{C}^\ast([0,T];X)
\end{split}
\end{equation*}
Observe that, thanks to PLICQ and to the complementarity conditions \eqref{necd1} we have also
\begin{equation*}
d\xi_m^j\to d\xi^j \qquad\text{in }\; \mathcal{C}^\ast([0,T];X).
\end{equation*}
The main point of the proof is showing that the sequence $\{ p_m : m\in\mathbb{N}\}$ is uniformly bounded in $W^{1,1}([0,T];X)$, so that a subsequence
of $\{ p_m\}$ will converge  weakly to a $BV$ function $p$. 
This fact, in turn, will imply that the further sequences $\{\psi_m\}$ and $\{\beta_m\}$
will converge (strongly) in the appropriate spaces, thanks to \eqref{pw} and \eqref{pr}. 
The convergence argument will be divided into three steps.

\smallskip

\noindent\textbf{Step 1}. The sequence $\{ p_m : m\in \mathbb{N}\}$ is bounded in $L^\infty ([0,T];X)$.\\
\textit{Proof of Step 1.} We start by rewriting \eqref{px} 
as 
\begin{equation}\label{pxx}
\begin{split}
p_m^{i+1} &= \big(\mathbb{I} - h_m D_x g(t_m^i,\bar{x}_m^i,\bar{w}_m^i)^\ast\big)p_m^i + \sum_{j=1}^\sigma \xi_m^{ij}\, x_\ast^j 
+\lambda_m h_m D_x g(t_m^i,\bar{x}_m^i,\bar{w}_m^i)^\ast\vartheta_m^{x,i}\\
      &\qquad +\lambda_m h_m\kappa_m^i(\bar{x}_m^0-\bar{x}(0)),\qquad i=0,\ldots,2^m-1,
\end{split}
\end{equation}
where we recall that $p_m^{2^m}=p_m^0$. Set $\gamma_m^i= | p_m^i|$, $i=0,\ldots,2^m-1$, $m\in\mathbb{N}$. By \eqref{p2mx} and \eqref{norm},
\begin{equation*}
\gamma_m^0\le 1\quad\forall m\in\mathbb{N}. 
\end{equation*}
By \eqref{pxx} we obtain 
\[
\gamma_m^1\leq \left( 1 + h_m L'\right) + \lambda_m h_m L' | \vartheta_m^{x,1}|
+\big| \sum_{j=1}^\sigma \xi_m^{1j}\, x_\ast^j \big|+\lambda_m h_m\Lambda=:d_m^1,
\]
and, for $i=2,\ldots,2^m-1$,
\[
\gamma_m^i\leq \left( 1 + h_m L'\right) \gamma_m^{i-1} + \lambda_m h_m L' | \vartheta_m^{x,i}|
  +\big| \sum_{j=1}^\sigma \xi_m^{ij}\, x_\ast^j \big|=:\left( 1 + h_m L'\right) \gamma_m^{i-1} +d_m^i.
\]
By induction, 
we obtain from the above conditions that, for each $k=1,\ldots,2^m-1$
\begin{equation*}
\gamma_m^k \le \sum_{i=1}^k d_m^i ( 1 + h_m L')^{k-i} = \sum_{\ell =0}^{k-1} d_m ^{k-\ell} (1 + h_m L')^\ell \le e^{TL'} \sum_{i=1}^k d_m^i, 
\end{equation*}
Therefore, for each $k=1,\ldots,2^m-1$, recalling \eqref{Lambda},
\begin{equation*}
\gamma_m^k\le e^{TL'} \left[ \lambda_m h_m (L'+1) \Lambda + \sum_{i=1}^{2^m} \big| \sum_{j=1}^\sigma \xi_m^{ij}\, x_\ast^j\big|\right].
\end{equation*}
Therefore, the sequence $\{\gamma_m^k : k=0,\ldots,2^m\}$ is 
bounded uniformly w.r.t.~$m$, and the proof of Step 1 is concluded.

\smallskip

\noindent \textbf{Step 2}. The sequence $\{ \dot{p}_m : m\in\mathbb{N}\}$ is bounded in $L^1(0,T;X)$ uniformly w.r.t.~$m$.\\
\textit{Proof of Step 2.} By \eqref{px} and Step 1,
\begin{equation*}
\| \dot{p}_m\|_{L^1}\le L' \lambda_m \Lambda + L' c + \big\| \sum_{j=1}^\sigma \xi_m^{j}\, x_\ast^j \big\|_{L^1(0,T;X)},
\end{equation*}
for a suitable constant $c$. Recalling \eqref{norm}, the claim follows.

\bigbreak

Up to taking another subsequence, by standard compactness results we can now assume that
\begin{equation*}
\begin{split}
p_m(t) &\to p(t)\quad\text{for all $t\in [0,T]$}\\
\dot{p}_m\dd t &\overset{\ast}{\rightharpoonup} \dd p\quad\text{in $\mathcal{C}^\ast([0,T];X)$},
\end{split}
\end{equation*}
for a suitable $BV$ function $p\colon[0,T]\to X$.

\smallskip

\noindent \textbf{Step 3}. The sequence $\{\psi_m\}$ converges strongly in $L^1(0,T;X)$ to a function $\psi$ that satisfies the weak maximality condition.
Furthermore, $\beta_m\to 0$ strongly in $L^1(0,T;X)$.\\
\textit{Proof of Step 3}. We obtain from \eqref{pw} 
that, for all $m$ and all $i=0,\ldots,2^m-1$, 
\begin{equation}\label{pw2}
\begin{split}
\frac{\psi_m^i}{h_m} &= -\omega_m^\ast + D_w g (t_m^i,\bar{x}_m^i,\bar{w}_m^i)^\ast p_m^i\\
&\quad + \lambda_m \big( \kappa_m(u_m^i-\bar{w}_m^i) -D_w g (t_m^i,\bar{x}_m^i,\bar{w}_m^i)^\ast\vartheta_m^{x,i} 
    - D_w f_2(t_m^i,\bar{w}_m^i)  \big)\in \NN_U(\bar{w}_m^i).
\end{split}
\end{equation}
We recall that, by Theorem \eqref{menouno}, $\kappa_m(\bar{w}_m-\bar{u})\to 0$ in $L^1(0,T;X)$. 
Therefore, recalling also Lemma \ref{Lemma1} and the above discussion of convergence of
$\vartheta_m$, $\omega_m^\ast$, $p_m$, the right-hand side of \eqref{pw2} converges strongly in $L^1(0,T;X)$ to
\begin{equation*}
-\omega + D_w g(t,\bar{x}(t),\bar{w}(t))^\ast p(t) -\lambda   D_w f_2(t,\bar{u}(t))=:\psi(t).
\end{equation*}
By the graph closedness of the normal cone $\NN_U$ we obtain also that $\psi(t)\in \NN_U(\bar{u}(t))$ a.e., hence concluding the proof of Step 3.

\smallskip

The above arguments also allow to pass to the limit along \eqref{px} and \eqref{p2mx} in the suitable topologies
and obtain the adjoint equation and the transversality condition.

We are therefore left with proving the nontriviality and the support conditions. To prove the first one, suppose by 
contradiction that both $\lambda$
and $p$ vanish. Then, by the adjoint equation, $\sum_{j=1}^\sigma x_\ast^j \dd \xi^j=0$. Therefore, by the weak maximality condition,
$\psi (t)\equiv -\omega$ is constant. Assume by contradiction that $\omega\neq 0$, so that necessarily $\bar{u}(t)\in\partial U$
for a.e.~$t$. Then, since $0\in \mathrm{int}\, U$,
\begin{equation*}
\langle\omega , \bar{u}(t)\rangle < \langle \omega , 0\rangle =0\quad\text{for a.e.~$t\in [0,T]$.} 
\end{equation*}
By integrating the above inequality we contradict the assumption that $\int_0^T\bar{u}(t)\dd t =0$. 
The above argument, therefore, shows that for all
$m$ large anough \eqref{norm} must be violated, hence concluding the proof of the nontriviality condition.

To prove the support condition, fix $j=1,\ldots,\sigma$ and set 
\begin{equation*}
E_j:= \{ t : \langle x_\ast^j,\bar{x}(t)\rangle < c^j\}. 
\end{equation*}
Assume that $E_j\neq\emptyset$ and let $K\subset E_j$ be compact. For all $m$ large enough, $\langle x_\ast^j,\bar{x}_m(t)\rangle < c_m^j(t)$
for all $t\in K$. By \eqref{px}, $\xi_m^j(t)=0$ on $K$, so that the support condition is proved.

The proof of Theorem \ref{nec} is concluded.
\end{proof}


\begin{thebibliography}{99}
	
	
\bibitem{Ago} D. Agostinelli, F. Alouges, A. DeSimone, Peristaltic waves as optimal gaits in metameric bio-inspired robots, 
Frontiers in Robotics and AI 5, 99 (2018)
\bibitem{arc} Ch. E. Arroud, G. Colombo, A Maximum Principle for the Controlled Sweeping Process, 
Set-Valued Var. Anal (2018) 26, 607-629.
\bibitem{AK} A. Arutyunov and D. Karamzin, A Survey on Regularity Conditions for State-Constrained Optimal Control Problems and the Non-degenerate Maximum Principle, J. Optim. Theory Appl. 184 (2020), 697-723.
\bibitem{AU} J-P. Aubin, Viability theory, Birkh\"auser (1991).
\bibitem{BMZ} F. Bagagiolo, R. Maggistro, M. Zoppello, Swimming by switching,
Meccanica 52(14): 3499–3511(2017)
\bibitem{Beh1} C. Behn, Adaptive control of straight worms without derivative measurement, Multibody Syst.
Dyn. 26 (2011) 213–243.
\bibitem{Beh2} C. Behn, Adaptive control of singularly perturbed worm-like locomotion systems, Differ. Equ.
Dyn. Syst. 21 (2013) 59–69.
\bibitem{BPZZ} N. Bolotnik, M. Pivovarov, I. Zeidis and K. Zimmermann, The undulatory motion of a chain of particles in a resistive medium, 
ZAMM Z. Angew. Math. Mech. 91 (2011) 259–275.
\bibitem{BK} M. Brokate and P. Krej\v{c}\'{\i}, Optimal control of ODE systems involving a rate independent variational inequality, 
Discrete and continuous dynamical systems series B. Volume 18 (2013), 331-348.
\bibitem{BKS} M. Brokate, P. Krej\v{c}\'{\i} and H. Schnabel, On uniqueness in evolution quasivariational inequalities
Journal of Convex Analysis 11, 2004, 111-130.
\bibitem{CHM} G. Colombo, R. Henrion, Nguyen D. Hoang, B. S. Mordukhovich, Optimal control of the sweeping 
process over polyhedral controlled sets, J. Differential Equations 260 (2016), 3397-3447.
\bibitem{gmd} G. Colombo, B. Sh. Mordukhovich, Nguyen Tr. Dao Nguyen, Optimization of a perturbed sweeping process by 
constrained discontinuous controls, submitted (2018), 26 pp.
\bibitem{CP} G. Colombo and M. Palladino, The minimum time function for the controlled Moreau's sweeping process, SIAM J. Control 54 
(2016), 2036-2062.
\bibitem{CGV} G. Colombo, P. Gidoni, and E. Vilches, Stabilization of periodic sweeping processes and  asymptotic average speed for soft locomotors with dry friction, in preparation.
\bibitem{CT} G. Colombo, L. Thibault, Prox-regular sets and applications, in  
\textit{Handbook of nonconvex analysis and applications}, 99-182, D. Y. Gao and D. Motreanu eds., Int. Press (2010).
\bibitem{DPFS} M.d.R. de Pinho, M.M.A. Ferreira, G.V. Smirnov, Optimal control involving sweeping processes. Set-Valued Var. Anal. 27 (2019), 523-548.
\bibitem{dSGN}A. DeSimone, P. Gidoni and G. Noselli, Liquid crystal elastomer strips as soft crawlers, 
Journal of the Mechanics and Physics of Solids  84, pp. 254-272 (2015), doi: 10.1016/j.jmps.2015.07.017
\bibitem{DesTat} A. DeSimone and A. Tatone, Crawling motility through the analysis of model locomotors: two case studies, Eur. Phys. J. E. 35 (2012).
\bibitem{Gid18} P. Gidoni, Rate-independent soft crawlers. Quart. J. Mech. Appl. Math. 71 (2018), 369-409.
\bibitem{GdS1} P. Gidoni and A. DeSimone, Stasis domains and slip surfaces in locomotion of a bio-inspired two-segment crawler, Meccanica 52 (2017) 587–601.
\bibitem{GdS2} P. Gidoni and A. DeSimone, On the genesis of directional friction through bristle-like mediating
elements, ESAIM Control Optim. Calc. Var. 23 (2017) 1023–1046.
\bibitem{Jung} K. Jung, J. C. Koo, et al. Artificial annelid robot driven by soft actuators, Bioinspiration Biomim. 2 (2007) S42–S49.
\bibitem{Kre} P. Krejčí, Hysteresis, Convexity and Dissipation in Hyperbolic Equations. Gattotoscho, 1996.
\bibitem{LM} C. Laschi, B. Mazzolai,
Lessons from Animals and Plants: The Symbiosis of Morphological Computation and Soft Robotics. IEEE Robot. Automat. Mag. 23(3): 107-114 (2016)
\bibitem{Mak} I. Gudoshnikov and O. Makarenkov, Structurally stable families of periodic solutions in sweeping processes of networks 
of elastoplastic springs. submitted.
\bibitem{Mak2} I. Gudoshnikov, M. Kamenskii, O. Makarenkov and N. Voskovskaia, One-period stability analysis of polygonal 
sweeping processes with application to an elastoplastic model, Mathematical Modelling of Natural Phenomena, in press.
\bibitem{MaThi} M. Mazade and L. Thibault, Regularization of differential variational inequalities with locally prox-regular sets. 
Math. Program. 139 (2013), Ser. B, 243-269. 
\bibitem{MieThe} A. Mielke and F. Theil, On rate-independent hysteresis models, NoDEA Nonlinear Differential Equations Appl. 11 (2004) 151–189.
\bibitem{MieRou} A. Mielke and T. Roubí\v{c}ek, Rate-independent Systems, Theory and Application. (Springer, New York 2015)
\bibitem{MdS} A. Montino and A. DeSimone, Dynamics and optimal actuation of a three-sphere low-Reynolds number
swimmer with muscle-like arms, Acta. Appl. Math. 149 (2017) 53–86.
\bibitem{BM} B. Sh. Mordukhovich, \textit{Variational Analysis and Generalized Differentiation, I: Basic Theory}, Springer (2006).
\bibitem{NdS} G. Noselli and A. DeSimone, A robotic crawler exploiting directional frictional interactions: 
Experiments, numerics and derivation of a reduced model, Proc. R. Soc. Lond. Ser. A Math. Phys. Eng. Sci. 470 (2014) 20140333.
\bibitem{RW} R. T. Rockafellar, R. J-B. Wets, \textit{Variational Analysis}, Springer (1998). 
\bibitem{SHB} M.  Schulke,  L.  Hartmann,  and  C.  Behn,  Worm-like  locomotion systems:  
Development  of  drives  and  selective  anisotropic  friction structures.  Proc.  56th  Int.  Scientific  Colloq.,  Ilmenau,  Germany, Sep. 2011.
\bibitem{vikas} V. Vikas, P. Grover and B. Trimmer, Model-free control framework for multi-limb soft robots,
2015 IEEE/RSJ International Conference on Intelligent Robots and Systems. (2015) 1111–
1116.
\bibitem{vinter} R.B. Vinter, {\it Optimal Control}, Birkh\"auser, Boston (2000).
\bibitem{WL} G. L. Wagner and E. Lauga, Crawling scallop: friction-based locomotion with one degree of freedom, J. Theoret. Biol. 324 (2013) 42–51.
\bibitem{ZC}  M. Zoppello and F. Cardin, Swim-like motion of bodies immersed in an ideal fluid ESAIM: COCV  25(16) (2019).


\end{thebibliography}
\end{document}